\documentclass[a4paper,11pt,twoside,reqno]{amsart}

\usepackage[T1]{fontenc}
\usepackage[utf8]{inputenc}
\usepackage{comment}

\usepackage{geometry}
\usepackage[english]{babel}

\usepackage{enumitem}
\setlist[enumerate,1]{label=(\alph*)}
\setlist[enumerate,2]{label=(\roman*)}
\setlist[enumerate]{font=\normalfont,leftmargin=2.8em}

\usepackage[lowtilde]{url}
\usepackage[unicode,bookmarks,colorlinks,breaklinks]{hyperref}
\hypersetup{linkcolor=blue,citecolor=blue,filecolor=black,urlcolor=blue}
\usepackage[capitalise]{cleveref}

\numberwithin{equation}{section}

\usepackage{amsmath}
\usepackage{mathtools}
\usepackage{amsthm}

\usepackage{amssymb}
\usepackage{mathrsfs} % For \mathscr

\usepackage{tikz-cd}

\usepackage{fixmath} % Italic uppercase greek by default. Recover roman with \mathrm

\theoremstyle{plain}
\newtheorem{theorem}{Theorem}[section]
\newtheorem{proposition}[theorem]{Proposition}
\newtheorem{lemma}[theorem]{Lemma}
\newtheorem{corollary}[theorem]{Corollary}

\theoremstyle{definition}
\newtheorem{definition}[theorem]{Definition}

\newtheorem{example}[theorem]{Example}
\theoremstyle{remark}
\newtheorem{remark}[theorem]{Remark}
\newtheorem{question}[theorem]{Question}

\newcounter{char}
\loop\ifnum\value{char}<27
  \expandafter\edef\csname\Alph{char}\Alph{char}\endcsname{\noexpand\mathbb{\Alph{char}}}
  \expandafter\edef\csname\Alph{char}bb\endcsname{\noexpand\mathbb{\Alph{char}}}
  \expandafter\edef\csname\Alph{char}ca\endcsname{\noexpand\mathcal{\Alph{char}}}
  \expandafter\edef\csname\Alph{char}sc\endcsname{\noexpand\mathscr{\Alph{char}}}
  \expandafter\edef\csname\Alph{char}fr\endcsname{\noexpand\mathfrak{\Alph{char}}}
  \expandafter\edef\csname\alph{char}fr\endcsname{\noexpand\mathfrak{\alph{char}}}
  \expandafter\edef\csname\Alph{char}bf\endcsname{\noexpand\mathbf{\Alph{char}}}
  \expandafter\edef\csname\alph{char}bf\endcsname{\noexpand\mathbf{\alph{char}}}
  \expandafter\edef\csname\Alph{char}rm\endcsname{\noexpand\mathrm{\Alph{char}}}
  \expandafter\edef\csname\alph{char}rm\endcsname{\noexpand\mathrm{\alph{char}}}
\addtocounter{char}{1}
\repeat

\renewcommand*{\AA}{\ifmmode\mathbb{A}\else{\r A}\fi}

\loop\ifnum\value{char}<27
  \expandafter\edef\csname b\Alph{char}\endcsname{\noexpand\mathbb{\Alph{char}}}
  \expandafter\edef\csname c\Alph{char}\endcsname{\noexpand\mathcal{\Alph{char}}}
\addtocounter{char}{1}
\repeat

\DeclareMathOperator{\codim}{codim}
\DeclareMathOperator{\cone}{cone}
\DeclareMathOperator{\Ho}{H}
\DeclareMathOperator{\Hom}{Hom}
\DeclareMathOperator{\id}{id}
\DeclareMathOperator{\im}{im}
\DeclareMathOperator{\Pic}{Pic}
\DeclareMathOperator{\Proj}{Proj}
\DeclareMathOperator{\RHom}{R\mathcal{H}om}

\newcommand*{\isomto}{\xrightarrow{\smash{\raisebox{-0.5ex}{\ensuremath{\scriptstyle\sim}}}}}

\newcommand*{\Category}[1]{\operatorname{\mathbf{#1}}}
\newcommand*{\Db}{\Category{D}^{\mathrm{b}}}
\newcommand*{\Dperf}{\Category{D}^{\mathrm{perf}}}
\newcommand*{\Dsg}{\Category{D}^{\mathrm{sg}}}

\newcommand*{\perf}{\mathrm{perf}}
\newcommand*{\sg}{\mathrm{sg}}

\newcommand*{\Canonical}{\omega}

\newcommand*{\tY}{\widetilde{Y}}
\newcommand*{\tX}{\widetilde{X}}

\DeclareMathOperator{\Serre}{\mathbb{S}}
\renewcommand*{\cA}{\mathscr{A}}
\renewcommand*{\cB}{\mathscr{B}}
\renewcommand*{\cC}{\mathscr{C}}
\renewcommand*{\cD}{\mathscr{D}}

\renewcommand*{\cT}{\mathscr{T}}
\renewcommand*{\cR}{\mathscr{R}}
\newcommand*{\Ideal}{\mathcal{I}}

\usepackage{xspace}
\usepackage{xpunctuate}

\newcommand*{\ie}{i.e.\@\xspace}
\newcommand*{\cf}{cf.\@\xspace}

\newcommand*{\loccit}{loc.~cit\xperiod}
\newcommand{\ignore}[1]{}

\begin{document}

\title{Kernels of categorical resolutions of nodal singularities}

% TODO Does this work after \begin{document}?
\hypersetup{
  pdfauthor={W. Cattani, F. Giovenzana, S. Liu, P. Magni, L. Martinelli, L. Pertusi, J. Song},
  pdftitle={Kernels of categorical resolutions of nodal singularities}
}

\author[W. Cattani]{Warren Cattani}
\author[F. Giovenzana]{Franco Giovenzana}
\author[S. Liu]{Shengxuan Liu}
\author[P. Magni]{Pablo Magni}
\author[L. Martinelli]{Luigi Martinelli}
\author[L. Pertusi]{Laura Pertusi}
\author[J. Song]{Jieao Song}

\address{W.~Cattani: SISSA, via Bonomea 265, 34136 Trieste, Italy}
\email{wcattani@sissa.it}
\urladdr{https://www.math.sissa.it/users/warren-cattani}
\address{F.~Giovenzana: Fakult\"at f\"ur Mathematik Technische Universit\"at Chemnitz,
Reichenhainer Stra\ss e 39, 09126 Chemnitz, Germany}
\email{franco.giovenzana@mathematik.tu-chemnitz.de}
\urladdr{https://sites.google.com/view/franco-giovenzana}
\address{S. Liu: Mathematical Institute, University of Warwick, Coventry, CV4 7AL, United Kingdom}
\email{Shengxuan.Liu.1@warwick.ac.uk}
\urladdr{https://warwick.ac.uk/fac/sci/maths/people/staff/sliu/}
\address{P. Magni: IMAPP, Radboud University Nijmegen, Heyendaalseweg 135, 6525 AJ Nijmegen, The Netherlands}
\email{p.magni@math.ru.nl}
\address{L. Martinelli: Laboratoire de Mathématiques d'Orsay, Université Paris-Saclay, Rue Michel Magat, Bât. 307, 91405 Orsay, France}
\email{luigi.martinelli@universite-paris-saclay.fr}
\address{L.~Pertusi: Dipartimento di Matematica F. Enriques, Universit\`a degli Studi di Milano, Via Cesare Saldini 50, 20133 Milano, Italy}
\email{laura.pertusi@unimi.it}
\urladdr{http://www.mat.unimi.it/users/pertusi}
\address{J.~Song: Université Paris Cité, CNRS, IMJ-PRG, F-75013 Paris, France}
\email{jieao.song@imj-prg.fr}
\urladdr{https://webusers.imj-prg.fr/~jieao.song/}

\keywords{Categorical resolutions, nodal singularities, cubic fourfolds}

\subjclass[2010]{14J45, 14J17, 18E30}

%------------------------------------------------------------------------------%
% Titlepage
%------------------------------------------------------------------------------%

\begin{abstract}
In this paper we study derived categories of nodal singularities. We show that for all nodal singularities there is a categorical resolution whose kernel is generated by a $2$ or $3$-spherical object, depending on the dimension. We apply this result to the case of nodal cubic fourfolds, where we describe the kernel generator of the categorical resolution as an object in the bounded derived category of the associated degree six K3 surface.

This paper originated from one of the problem sessions at the Interactive Workshop and Hausdorff School ``Hyperkähler Geometry'', Bonn, September 6-10, 2021.
\end{abstract}

\maketitle

\setcounter{tocdepth}{1}
\tableofcontents

%------------------------------------------------------------------------------%
% Section: Introduction
%------------------------------------------------------------------------------%

\section{Introduction}

Resolution of singularities is a central topic studied in algebraic geometry. Since Hironaka \cite{Hir} proved that singularities of varieties in characteristic 0 can be resolved, there has been much progress in studying singularities, their resolutions, and their applications in birational geometry. On the other hand, derived categories provide a strong technique for understanding algebraic varieties, for example two smooth Fano (or general type) varieties with equivalent derived categories are isomorphic \cite{BOreconstruction}. For other varieties, derived categories can yield information about their birational geometry, for example flops of three dimensional varieties induce derived equivalences of their derived categories \cite{Bfd}.

One can often study a singularity by considering the properties of a resolution of it, and for relatively simple varieties and singularities, this might be done concretely. From the categorical viewpoint, let $Y$ be a singular variety and let $\sigma\colon \tY\rightarrow Y$ be a resolution of singularities, then we have derived functors between their derived categories
\[
  \sigma^*\colon \Dperf(Y)\rightarrow \Dperf(\tY),\quad \sigma_*\colon \Db(\tY)\rightarrow \Db(Y).
\]
Since $\tY$ is a smooth variety, we have $\Db(\tY)=\Dperf(\tY)$. The two functors are related by the projection formula
\[
  \sigma_*\sigma^*(\Fca)=\Fca \otimes \sigma_*\cO_{\tY}.
\]

Inspired by the geometric picture, Kuznetsov introduced in \cite{Kldcr} the definition of ``abstract'' categorical resolution of singularities (see \cref{def:cr}). In the case of $\Db(Y)$, it consists of a triple $(\widetilde{\cD}, \sigma_*, \sigma^*)$, where $\widetilde{\cD}$ is a geometric triangulated category, $\sigma_* \colon \widetilde{\cD} \to \Db(Y)$ and $\sigma^* \colon \Dperf(Y) \to \widetilde{\cD}$ are functors such that $\sigma^*$ is left adjoint to $\sigma_*$, and the natural morphism of functors $\id_{\cD^{\perf}}\rightarrow \sigma_*\sigma^*$ is an isomorphism.

Now an interesting question is whether or not this categorical viewpoint allows one to characterize the singularity geometrically. To shed some light on this, we investigate in this paper one special kind of singularities and their categorical resolutions, namely \emph{nodal singularities}.

Before stating our main result, we briefly recall a few notions. A resolution of singularities is crepant if its relative canonical class is trivial. Crepant resolutions are interesting since they are considered minimal resolutions in the case of Gorenstein varieties, but they are also rare. On the other hand, a categorical resolution of singularities $\sigma_*\colon \widetilde{\cD}\rightarrow \Db(Y)$ is \emph{weakly crepant} if the left adjoint $\sigma^*$ of $\sigma_*$ is also its right adjoint (see \cref{def:cr}). An object $\Tca\in\widetilde{\cD}$ is called $k$-\emph{spherical} if $\Hom^\bullet(\Tca,\Tca)=\mathbb{C}\oplus\mathbb{C}[-k]$ and there is an isomorphism of functors $\Hom(\Tca, -) = \Hom(-,\Tca[k])^\vee$; and $\Eca\in\widetilde{\cD}$ is \emph{exceptional} if $\Hom^\bullet(\Eca,\Eca)=\mathbb{C}$. 

\begin{theorem}\label{thm:ns}
Let $Y$ be a quasiprojective variety with an isolated nodal singularity, and assume $\dim(Y)\geq2$. Then there exists a weakly crepant categorical resolution $\sigma_*\colon \widetilde{\cD}\rightarrow \Db(Y)$ such that:
\begin{enumerate}
  \item The kernel $\ker(\sigma_*)$ of $\sigma_*$ is classically generated by a single object $\Tca$ which is $2$-spherical if $\dim(Y)$ is even, and $3$-spherical otherwise.
  \item The resolution $\sigma_*$ is a localization functor up to direct summands, \cf\cref{def:localiz-functor}. 
\end{enumerate}
\end{theorem}

Note that the existence of a weakly crepant categorical resolution is a direct application of \cite{Kldcr}. We remark that the constructed categorical resolution has the advantage of being weakly crepant in any dimension, while the geometric resolution $\Db(\tY)$ is not. In \cite{Kldcr} another notion of crepancy, called strong crepancy, was introduced. The resolution $\widetilde{\cD}$ in \cref{thm:ns} is not strongly crepant, as computed in \cref{prop:notstrcrep}, if the dimension of $Y$ is bigger than $3$.

Our contribution is the explicit description of the kernel of the categorical resolution. We will define the resolution~$\widetilde{\cD}$ as an admissible component of $\Db(\tY)$, where $\widetilde{Y}$ is the blow-up at the isolated nodal singularity. Here, the object $\Tca$ has a clear geometric meaning: if $\dim(Y)$ is even, the object $\Tca$ is the pushforward to $\tY$ of the spinor bundle on the quadric exceptional divisor; and if $\dim(Y)$ is odd, the object $\Tca$ is described as the right mutation of the pushforward of one of the spinor bundles through the other, see \cref{prop:catresker}.

\begin{remark}
\cref{thm:ns} has been recently proven independently by Kuznetsov and Shinder in \cite[Theorem~5.8]{KuzSh} with a similar strategy. Furthermore, \cite[Theorem~5.2]{KuzSh} explains that one can drop ``classically'' and ``up to direct summands'' in \cref{thm:ns}; see also \cite{MS:22} for a discussion about this. Finally, note that the case when $Y$ is $1$-dimensional has been recently studied in \cite{sung2022remarks}.
\end{remark}

Based on \cref{thm:ns}, we propose the following definitions of categorical nodal singularities.
\begin{definition}[(Abstract) nodal category]\label{def:anc}
A triangulated category $\cT$ is called \emph{(abstract) nodal} if there is a categorical resolution $\sigma_*\colon \widetilde{\cD}\rightarrow \cT$ which is weakly crepant and whose kernel is (classically) generated by a single spherical object.
\end{definition}

\begin{definition}[Geometric nodal category]
A triangulated category $\cT$ is called \emph{geometric nodal} if it is an admissible subcategory of the derived category $\Db(Y)$ of a normal quasiprojective variety $Y$ which has only an isolated nodal singularity, such that $\cT^\perf$ is not smooth\footnotemark.
\end{definition}
\footnotetext{When $\cT$ is a triangulated category, we say in this article that $\cT^{\perf}$ is \emph{smooth} if $\cT$ can be realized as an admissible subcategory of the bounded derived category $\Db(X)$ of a smooth variety $X$. This means in particular that $\cT^{\perf}=\cT$ by \cite[Proposition~1.10]{Orl:06} and the fact that $\Dperf(X)=\Db(X)$.}

Using \cref{thm:ns} we show the following relation between the above definitions.

\begin{theorem}[\cref{thm:geom-nodal-implies-abstract-nodal2}]\label{thm:geom-nodal-implies-abstract-nodal}
If $\cT$ is a geometric nodal category, then $\cT$ is an abstract nodal category.
Furthermore, the constructed categorical resolution $\sigma_*\colon\widetilde{\cD}\to\cT$ as in the definition of an abstract nodal category is a localization up to direct summands.
\end{theorem}

However, there are some questions around the definition of abstract nodal category. 
\begin{question}\label{qst:problems}~
\begin{enumerate}
    \item The sphericalness property depends on the dimension of the variety. What should be a suitable definition of dimension of an abstract triangulated category?
    \item It is not clear to us whether the definition characterizes nodal singularities in the geometric picture. In other words, if $Y$ is a variety such that $\Db(Y)$ is abstract nodal, is then $\Db(Y)$  a geometric nodal category?
    \item Does the sphericalness property of the kernel generator already imply that the resolution is weakly crepant?
    \item Suppose that there is a $2$ or $3$-spherical object $\Tca$ in $\Db(X)$ where $X$ is a smooth projective variety, and let $\cT\subset\Db(X)$ be the triangulated subcategory classically generated by $\Tca$.
    Is the quotient $\Db(X)/\cT$ a \emph{geometric} nodal category?
\end{enumerate}
\end{question}

\begin{remark}
 Note that a positive answer to \cref{qst:problems}.(c) has been recently given in \cite[Lemma 5.8]{KS:22b}. 
\end{remark}

To address the last problem above, we study a concrete example: Let $Y \subset \PP^5$ be a nodal cubic fourfold, with hyperplane section class $H$. By \cite{Kcf} there is a semiorthogonal decomposition of $\Db(Y)$ given by
\[
  \Db(Y)=\langle\cA_Y,\cO_Y,\cO_Y(H),\cO_Y(2H)\rangle,
\]
where $\cA_Y\coloneqq\langle \cO_Y,\cO_Y(H),\cO_Y(2H) \rangle^{\perp}$ and $\cO_Y,\cO_Y(H),\cO_Y(2H)$ form an exceptional collection of line bundles. Then a categorical resolution of $\cA_Y$ is provided by $\Db(S)$, where $S$ is a K3 surface of degree $6$ obtained as the intersection in $\PP^4$ of a smooth quadric hypersurface $Q$ with a cubic hypersurface. In this situation, we have the following application of \cref{thm:ns}, which provides an answer to \cite[Remark~5.9]{Kcf}.

\begin{theorem} \label{prop:nodalc4}
If $Y$ is a nodal cubic fourfold, then the kernel of the categorical resolution $\Db(S) \to \cA_Y$ is classically generated by $t^*\cS$, where $t\colon S\rightarrow Q$ is the inclusion map of the K3 surface into the defining quadric $Q$ of $S$, and $\cS$ denotes the spinor bundle on $Q$.
\end{theorem}

\begin{remark}
Note that the object $t^*\cS \in \Db(S)$ is 2-spherical.
This is similar to the situation of a nodal K3 surface, in which the spherical objects $\cO_{E_i}(-1)$ appear in the kernel, where $E_i$ are the exceptional curves in the resolution, \cf\cite[Lemma~2.3]{Kuz:21} and \cite[Lemma~3.1]{Bfd}. 
\end{remark}

\subsection{Plan of Paper}
In \cref{sec:preliminary}, we recall the definitions and theorems that we use in the following sections. In particular, we review the definitions and properties of nodal singularities, the construction of their  categorical resolution via a Lefschetz decomposition following \cite{Kldcr}, and some results in \cite{Ehf} which we use to compute the kernel of these categorical resolutions.

\cref{sec:generalcase} is about the proof of \cref{thm:ns}. We first use a Lefschetz decomposition of quadrics to construct a categorical resolution of varieties with an isolated nodal singularity as in \cite{Kldcr}. Then by results of \cite{Ehf}, we find the kernel generator and check the sphericalness property.  

In \cref{sec:nodalcubic4}, we focus on the case of nodal cubic fourfolds, proving \cref{prop:nodalc4} as a consequence of Theorem \ref{thm:ns}. 

\subsection{Notations and Conventions}

By variety we mean an integral scheme that is separated and of finite type over $\mathbb{C}$. If not otherwise mentioned, all functors between derived categories are implicitly derived. We use $\mathbb{R}_\cA$ and $\mathbb{L}_\cA$ to denote the right and left mutation with respect to an admissible subcategory $\cA$, and use $\mathbb{T}_B$ to denote the twist functor $-\otimes B$. We define $\Hom^\bullet(-,-)=\bigoplus_{i}\Hom(-,-[i])[-i]$. If $\cT$ is a triangulated category, a \emph{classical generator} of $\cT$ is an object $\Tca\in\cT$ such that the smallest  strictly full triangulated subcategory of $\cT$ which is closed under direct summands and containing $\Tca$ is equal to $\cT$, in symbols $\cT=\langle\Tca\rangle^{\oplus}$.
We take the liberty to write most isomorphisms as equalities.

\subsection*{Acknowledgments} This project started at the Interactive Workshop and Hausdorff School ``Hyperkähler Geometry'' at the Mathematical Institute of the University of Bonn in September 2021. It is our pleasure to thank this institution and the organizers of the workshop for the very stimulating atmosphere. We are also grateful to Alex Perry who proposed the question of computing the kernel of the categorical resolution of a nodal cubic fourfold, motivating our working group. We thank Federico Barbacovi, Thorsten Beckmann, Daniel Huybrechts, Chunyi Li, Emanuele Macr\`i for interesting conversations related to this work. We also thank Alexander Kuznetsov and Evgeny Shinder for helpful comments on a draft of this article, and the referee for careful reading of the paper.

W.C.\ was supported by the national research project PRIN 2017 Birational geometry and moduli spaces. F.G.\ was supported by the DFG through the research grant Le 3093/3-2. S.L.\ is funded by the CSC-Warwick scholarship and Royal Society URF$\backslash$R1$\backslash$201129 ``Stability condition and application in algebraic geometry”. P.M.\ was supported by the Netherlands Organization for Scientific Research (NWO) under project number 613.001.752. L.M.\ was supported by the ERC Synergy Grant 854361 HyperK. L.P.\ was supported by the national research project PRIN 2017 Moduli and Lie Theory. J.S.\ was supported by a doctoral contract at Université Paris Cité.

{\Small\subsection*{Statements \& Declarations}

\noindent \textit{Competing interests.}  The authors read and 
approved the final manuscript. The author has no relevant financial or 
non-financial interests to disclose.

\noindent \textit{Availability of data and material.} This paper has 
no associated data and material.}

%------------------------------------------------------------------------------%
% Section: Preliminaries
%------------------------------------------------------------------------------%

\section{Preliminaries} \label{sec:preliminary}
In this section, we briefly recall the notation and tools that we will use in subsequent sections. In particular, we discuss nodal singularities, semiorthogonal decompositions, categorical resolutions arising from Lefschetz decompositions and some results from \cite{Ehf} that allow to compute the kernel of certain categorical resolutions. Finally we review some properties of spinor bundles on quadrics, and perform some cohomology computations we need in later sections.
\subsection{Nodal singularities} \label{subsec:nodalsing}
Let $X$ be a variety of dimension $n$. We recall the definition of a nodal singularity, which is the simplest kind of hypersurface singularity.
\begin{definition}\label{def:nodalpoint}
An isolated singular point $x\in X$ is a \emph{nodal point} (or \emph{ordinary double point}) if the variety $X$ is formally locally around $x$ isomorphic to the singularity defined by the origin of the zero locus of $x_0^2+x_1^2+x_2^2+\dots+x_n^2$ inside $\bA_\mathbb{C}^{n+1}$, \ie $\widehat{\cO_{X,x}}\simeq\CC[\![x_0,\dots,x_n]\!]/(x_0^2+\dots+x_n^2)$.
\end{definition}

\begin{remark}
Since we are working over $\bC$, we can replace ``formally locally'' with ``analytically locally'' and obtain an equivalent definition. Indeed, the completions of the algebraic and the analytic local rings coincide, \cf \cite[Proposition~3]{GAGA}, and two analytic germs are equivalent if and only if the completions of their analytic local rings are isomorphic, \cf \cite[Theorem~4.2.3]{ishii}.
\end{remark}

Assume that $X$ has only one nodal singularity $x\in X$ and is smooth elsewhere. 
Since hypersurface singularities are Gorenstein, so is $X$ (recall that being Gorenstein can be checked after completion of local rings, \cf\cite[Proposition~3.1.19.(c)]{BH:98}). Now let  $\sigma\colon \widetilde{X}\rightarrow X$ be the blow-up of $X$ at $x$. Then $\sigma$ is a resolution of singularities whose exceptional locus $j\colon Q\rightarrow \widetilde{X}$ is the smooth projective quadric hypersurface defined by $x_0^2+x_1^2+x_2^2+...+x_n^2$.
The conormal bundle of $Q\subset\tX$ is $\cO_Q(1)=\cN^{\vee}_{Q/\tX}= j^*\cO_{\widetilde{X}}(-Q)$, since $Q$ is the exceptional Cartier divisor of a blow-up.

\subsection{Semiorthogonal decompositions} We recall the definitions of admissible subcategories and exceptional collections, which are the main source of semiorthogonal decompositions.
Denote by $\cT$ a triangulated category. 

\begin{definition}[Admissible subcategory]
Let $\cA\subset \cT$ be a full triangulated subcategory. Then $\cA$ is called \emph{admissible} if the embedding functor of $\cA$ into $\cT$ admits left and right adjoints.
\end{definition}

\begin{definition}[Semiorthogonal decomposition]
Let $\cA_1,\dots,\cA_m$ be a sequence of admissible subcategories of $\cT$. Then we say that $\cA_1,\dots,\cA_m$ is a \emph{semiorthogonal collection} if $\Hom(\cA_i,\cA_j)=0$ for all $i>j$. If in addition this collection generates $\cT$, we say that it forms a \emph{semiorthogonal decomposition} of $\cT$, which we denote by 
\[
  \cT=\langle \cA_1,\dots,\cA_m\rangle.
\]
\end{definition}
Any admissible subcategory $\cA$ induces a semiorthogonal decomposition: Set
\begin{align*}
    \cA^\perp&=\{\Fca\in \cT\mid \Hom(\cA,\Fca)=0\},\\
    {}^{\perp}\cA&=\{\Fca\in \cT\mid \Hom(\Fca,\cA)=0\},
\end{align*}
then there are two semiorthogonal decompositions
\begin{equation*}
    \cT=\langle \cA^\perp,\cA \rangle, \quad\quad \cT=\langle \cA, {}^{\perp}\cA \rangle.
\end{equation*}
We define the left mutation functor $\bL_{\cA}$ and the right mutation functor $\bR_{\cA}$ to fit into the following exact triangles, respectively,
\[
  \alpha\alpha^!\rightarrow \id\rightarrow \bL_{\cA},\quad\quad \bR_{\cA}\rightarrow \id\rightarrow \alpha\alpha^*,
\]
where $\alpha\colon \cA\rightarrow\cT$ is the embedding functor and $\alpha^!$ and $\alpha^*$ are its right and left adjoints, respectively. Note that the semiorthogonality ensures that the cones in the triangles above are functorial. Indeed, use that the decomposition of an object of $\cT$ into semiorthogonal components is functorial to deduce that $\im(\bR_{\cA})\subset {}^{\perp}\cA$, and then consider the long exact sequences arising from applying, for example, $\Hom(\bR_{\cA}(\Fca),-)$ to the triangles above. The following lemmata describe the interaction between mutation functors and semiorthogonal decompositions.

\begin{lemma}{\cite[Corollary~2.9]{Kcf}}\label{SOD-mutations}
Assume that $\cT=\langle \cA_1,\dots,\cA_m\rangle$ is a semiorthogonal decomposition. Then for each $1 \le i \le m-1$ there is a semiorthogonal decomposition
\[
\cT=\langle \cA_1,\dots,\cA_{i-1}, \bL_{\cA_i}(\cA_{i+1}), \cA_i,\cA_{i+2}, \dots,\cA_m\rangle
\]
and for each $2 \le i \le m$ there is a semiorthogonal decomposition
\[
\cT=\langle \cA_1,\dots,\cA_{i-2},\cA_i, \bR_{\cA_i}(\cA_{i-1}),\cA_{i+1}, \dots,\cA_m\rangle.
\]
\end{lemma}

\begin{lemma}{\cite[Lemma~2.2]{Kcy}}\label{mutations-SOD}
Let $\cA$ be an admissible subcategory of $\cT$. Assume that $\cA$ admits a semiorthogonal decomposition $\cA=\langle \cA_1,\dots,\cA_m\rangle$. Then
\[
\bL_\cA = \bL_{\cA_1} \circ \cdots \circ \bL_{\cA_m} \quad \text{and} \quad \bR_\cA = \bR_{\cA_m} \circ \cdots \circ \bR_{\cA_1}.
\]
\end{lemma}

Examples of admissible subcategories are given by exceptional objects.

\begin{definition}[Exceptional object]
An object $\Eca\in \cT$ is \emph{exceptional} if $\Hom^\bullet(\Eca,\Eca)=\CC$.\footnotemark
\footnotetext{If the category $\cT$ is not proper, one also requires that the functors $\Hom^\bullet(\cE, -)$ and $\Hom^\bullet(-, \cE)$ on $\cT$ take values in the category of finite-dimensional graded vector spaces.}
\end{definition}

\begin{definition}[Exceptional collection]
A set of objects $\lbrace \Eca_1, \dots, \Eca_m \rbrace$ in $\cT$ is an \emph{exceptional collection} if each $\Eca_i$ is exceptional, and $\Hom^\bullet(\Eca_i, \Eca_j)=0$ when $i>j$.
\end{definition}

If $\Eca$ is an exceptional object in a triangulated category $\cT$, then the full triangulated subcategory $\cA=\langle \Eca\rangle$ generated by $\Eca$ is admissible, \cf \cite{bondal1990representable}; the mutations of an object $\Fca\in \cT$ can be described explicitly as
\begin{equation*}
    \bL_\Eca (\Fca) = \cone (\Hom^\bullet(\Eca,\Fca)\otimes \Eca\rightarrow \Fca), \quad\quad \bR_\Eca (\Fca) = \cone (\Fca\rightarrow \Hom^\bullet(\Fca,\Eca)^{\vee}\otimes \Eca)[-1].
\end{equation*}
Similarly, an exceptional collection gives rise to a semiorthogonal collection.

In this paper, we consider a special kind of semiorthogonal decomposition.

\begin{definition}[Lefschetz decomposition]{\cite[Definition~2.16]{Kldcr}}
Let $X$ be a variety with a (not necessarily ample) line bundle $\cO(1)$.
A \emph{Lefschetz decomposition} of $\Db(X)$ is a semiorthogonal decomposition of the form 
\begin{equation*}
    \Db(X)=\langle \cB_0, \cB_1(1), \dots , \cB_{m-1}(m-1)\rangle \quad \text{where} \quad 0\subset \cB_{m-1}\subset \dots \subset \cB_1 \subset \cB_0 \subset \Db(X).
\end{equation*}
A Lefschetz decomposition is \emph{rectangular} if $\cB_0=\cB_1=\dots=\cB_{m-1}$.
Similarly, a \emph{dual Lefschetz decomposition} is a semiorthogonal decomposition of the form
\begin{equation*}
    \Db(X)=\langle \cB_{m-1}(1-m),\dots, \cB_{1}(-1), \cB_0 \rangle \quad \text{where} \quad 0\subset \cB_{m-1}\subset \dots \subset \cB_1 \subset \cB_0 \subset \Db(X).
\end{equation*}
\end{definition}

\subsection{Spherical objects and Serre functors}
Let $\cT$ be a triangulated category. We recall the definition of spherical objects, which play an important role in this paper.
\begin{definition}[$k$-spherical object]{\cite[Definition~2.14, Lemma~2.15]{STBgadccs}}\label{def:spherical with Serre}
Let $k\in\bN$ be a natural number. An object $\Tca \in \cT$ is called $k$-\emph{spherical} if
\begin{enumerate}
\item the functors $\Hom^\bullet(\Tca, -)$ and $\Hom^\bullet(-, \Tca)$ on $\cT$ take values in the category of finite-dimensional graded vector spaces;
\item $\Hom^\bullet(\Tca,\Tca)=\mathbb{C}\oplus\mathbb{C}[-k]$;
\item for any $\Fca \in \cT$ there is an isomorphism $\Hom(\Tca,\Fca)=\Hom(\Fca, \Tca[k])^\vee$, which is functorial in $\Fca$.
\end{enumerate}
\end{definition}

Condition (c) in \cref{def:spherical with Serre} can be simplified in some situations, for instance when $\cT$ has a Serre functor.

\begin{definition}[Serre functor]
Let $\cT$ be a triangulated category. 
An equivalence $\Serre\colon \cT\rightarrow\cT$ is called \emph{Serre functor} if for any two objects $\Fca,\Gca\in\cT$ there is a bifunctorial isomorphism
\[
  \Hom(\Fca,\Gca)=\Hom(\Gca,\Serre(\Fca))^{\vee}.
\]
\end{definition}

For instance, by Grothendieck--Verdier duality \cite[Theorem~3.34]{huybrechts2006fourier} the Serre functor of the derived category $\Db(X)$ of a smooth projective variety $X$ of dimension $n$ is given by $\mathbb{T}_{\Canonical_X} \circ [n]$, where $\Canonical_X$ is the canonical bundle of $X$. The Serre functor is unique up to isomorphisms of exact functors.
The following lemma describes the relation between Serre functors and semiorthogonal decompositions with two components.

\begin{lemma}\label{lemma:serre1}{\cite[Lemma~2.11]{Kcf},\cite[Lemma~2.6]{Kcy}}
Let $\cT=\langle\cA,\cB\rangle$ be a semiorthogonal decomposition of a triangulated category. Assume that $\cT$ has Serre functor $\Serre_{\cT}$. Then
\begin{enumerate}
    \item there are semiorthogonal decompositions $\cT=\langle\Serre_{\cT}(\cB),\cA\rangle = \langle \cB,\Serre^{-1}_{\cT}(\cA)\rangle$, and
    \item $\cA$ and~$\cB$ have Serre functors $\Serre_{\cA}$ and $\Serre_{\cB}$, respectively, satisfying the relations
\[
  \Serre_{\cB} = \bR_{\cA}\circ \Serre_{\cT}, \quad\quad \Serre^{-1}_{\cA}=\bL_\cB\circ \Serre^{-1}_{\cT}.
\]
\end{enumerate}

\end{lemma}
\begin{remark}\label{rem:spherical with serre} Assume that $\cT$ admits a Serre functor $\Serre$. By the Yoneda lemma, in the \cref{def:spherical with Serre} of a $k$-spherical object $\Tca \in \cT$ we can replace condition (c) with $\Serre(\Tca)=\Tca[k]$.
\end{remark}

This paper is about varieties with an isolated nodal singularity. By the local nature of such singularities, it seems unnatural to focus just on projective varieties; we prefer instead to work with quasiprojective varieties. The smooth varieties arising from resolution of singularities will again be quasiprojective; in particular, their derived category will not have a Serre functor, but they will admit the following weaker version.

\begin{definition}[Serre functor for a pair $(\cR, \cT)$]{\cite[Section~6.4]{VdBNccr}}\label{def:Serre functor pair}
Let $\cT$ be a triangulated category. Let $\cR \subset \cT$ be a full triangulated subcategory such that for any $\Fca \in \cR$ the functors $\Hom^\bullet(\Fca, -)$ and $\Hom^\bullet(-, \Fca)$ on $\cT$ take values in the category of finite-dimensional graded vector spaces. An equivalence $\Serre \colon \cT\rightarrow\cT$ is called \emph{Serre functor for the pair $(\cR, \cT)$} if
\begin{enumerate}
\item $\Serre$ leaves $\cR$ stable and
\item for any two objects $\Fca \in \cR, \Gca\in\cT$ there is a bifunctorial isomorphism
\[
  \Hom(\Fca,\Gca)=\Hom(\Gca,\Serre(\Fca))^{\vee}.
\]
\end{enumerate} 
In particular, the restriction of $\Serre$ to $\cR$ is a Serre functor for $\cR$.
\end{definition}

\begin{example}\label{ex:serre quasiproj}
Let $X$ be a smooth quasiprojective variety of dimension $n$. Let $j \colon E \to X$ be the embedding of a smooth projective divisor; denote by $\omega_j \coloneqq \omega_E \otimes j^* \omega_X^\vee$ its relative dualizing bundle. Define the category $\Db_E(X)$ as the full subcategory of $\Db(X)$ consisting of complexes topologically supported on $E$. As a triangulated category, $\Db_E(X)$ is generated by $j_*\Db(E)$, a remark that is very useful in practice. For any $\Fca \in \Db_E(X)$, the functors $\Hom^\bullet(\Fca, -)$ and $\Hom^\bullet(-, \Fca)$ take values in the category of finite-dimensional graded vector spaces: indeed, this holds true for an object of the form $j_*\Fca,\Fca \in \Db(E)$, because $\Hom ^\bullet (-, j_*\Fca) = \Hom^\bullet  (j^*(-),\Fca)$ and $\Hom^\bullet  (j_*\Fca,-) = \Hom^\bullet  (\Fca, j^*(-) \otimes \omega_j[-1])$. We claim that $\mathbb{T}_{\Canonical_X} \circ [n]$ is a Serre functor for the pair $(\Db_E(X),\Db(X))$. Condition (a) in \cref{def:Serre functor pair} is clearly satisfied; as for condition (b), for any $\Fca \in \Db(E)$ and $\Gca \in \Db(X)$, by Grothendieck-Verdier duality we have
\begin{align*}
\Hom_X(j_*\Fca, \Gca)
=& \Hom_E(\Fca, j^*\Gca \otimes \omega_j[-1]) \\
=& \Hom_E(\Fca, j^*\Gca \otimes \omega_E \otimes j^* \omega_X^\vee[-1]) \\
=& \Hom_E(j^*(\Gca \otimes \omega_X^\vee)[-n], \Fca)^\vee \\
=& \Hom_X(\Gca, j_*\Fca \otimes \omega_X[n])^\vee.
\end{align*}
\end{example}
The following result is analogous to \cref{lemma:serre1}, and can be proven in the same way.
\begin{lemma}\label{lemma:Serre pair}
Let $\cT$ be a triangulated category, and $\cR \subset \cT$ a full triangulated subcategory. Suppose that we have a full triangulated subcategory $\cA$ of $\cR$ that is admissible in both $\cR$ and $\cT$; in particular, we have semiorthogonal decompositions $\cR=\langle\cA,\cB\rangle$ and $\cT=\langle\cA,\cC\rangle$. Assume that the pair $(\cR, \cT)$ has a Serre functor $\Serre_{\cR, \cT}$. Then the pair $(\cB, \cC)$ has a Serre functor, which is given by
\[
 \Serre_{\cB, \cC} = \bR_{\cA}\circ \Serre_{\cR,\cT}.
\]
\end{lemma}
\begin{remark} Assume that an object $\Tca \in \cT$ belongs to a full triangulated subcategory $\cR \subset \cT$ such that the pair $(\cR,\cT)$ has a Serre functor $\Serre$. To check that $\Tca$ is $k$-spherical, condition (c) in \cref{def:spherical with Serre} can be replaced (again by the Yoneda lemma) with $\Serre(\Tca)=\Tca[k]$.
\end{remark}

\subsection{Categorical resolutions}

We recall the material from \cite[§3]{Kldcr}.

\begin{definition}[Geometric category]
A triangulated category $\cD$ is \emph{geometric} if it is equivalent to an admissible subcategory of $\Db(X)$, where $X$ is a smooth variety.
\end{definition}

\begin{definition}\cite[Definition~1.6]{Orl:06}\label{def:hom-finite}
Let $\cD$ be a triangulated category.
An object $\Fca\in\cD$ is \emph{homologically finite} if for any $\Gca\in \cD$ there exists only a finite number of $n\in\mathbb{Z}$ such that $\Hom_{\cD}{(\Fca,\Gca[n])}\neq 0$.
The category $\cD^{\perf}$ is defined as the full subcategory of $\cD$ whose objects are the homologically finite objects.
\end{definition}

\begin{remark}
The notation $\cD^{\perf}$ is justified since the homologically finite objects in the bounded derived category of coherent sheaves on a quasiprojective variety $X$ are nothing else than the perfect complexes, \ie $\Db(X)^{\perf}=\Dperf(X)$, \cf\cite[Proposition~1.11]{Orl:06}.
\end{remark}

\begin{definition}[Categorical resolution] \label{def:cr}
A \emph{categorical resolution} of a triangulated category $\cD$ is a geometric triangulated category $\widetilde{\cD}$ and a pair of functors 
\begin{equation*}
    \sigma_*\colon  \widetilde{\cD}\rightarrow \cD,\quad\quad \sigma^*\colon \cD^{\perf}\rightarrow \widetilde{\cD},
\end{equation*}
such that $\sigma^*$ is left adjoint to $\sigma_*$ on $\cD^{\perf}$, \ie
\begin{equation*}
    \Hom_{\widetilde{\cD}}{(\sigma^*\Fca, \Gca)} = \Hom_{\cD}{(\Fca, \sigma_*\Gca)} \quad\quad \text{for any} \quad \Fca \in \cD^{\perf},\ \Gca \in \widetilde{\cD},
\end{equation*}
and the natural morphism of functors $\id_{\cD^{\perf}}\rightarrow \sigma_*\sigma^*$ is an isomorphism.

A categorical resolution $(\widetilde{\cD}, \sigma_*,\sigma^*)$ is \emph{weakly crepant}  if $\sigma^*$ is also right adjoint to $\sigma_*$ on $\cD^{\perf}$, \ie
\begin{equation*}
    \Hom_{\widetilde{\cD}}{(\Gca, \sigma^*\Fca)} = \Hom_{\cD}{(\sigma_*\Gca, \Fca)} \quad\quad \text{for any} \quad \Fca \in \cD^{\perf},\ \Gca \in \widetilde{\cD}.
\end{equation*}
\end{definition}

We now focus on a particular construction of a (weakly crepant) categorical resolution starting from a classical resolution of singularities. Consider a resolution of rational singularities $\sigma\colon \tY\rightarrow Y$ whose exceptional locus $E$ is an irreducible divisor. Let $Z=\sigma(E)$ and $\rho\colon E\rightarrow Z$ be the restriction of $\sigma$ to $E$. Denote by $j\colon E \rightarrow \tY$ the inclusion morphism. Let 
\begin{equation}\label{eq:dual Lef}
    \Db(E)=\langle \cB_{m-1}(1-m),\dots, \cB_1(-1),\cB_0 \rangle
\end{equation}
be a dual Lefschetz decomposition with respect to $\cO_{E}(1) \coloneqq \cN^{\vee}_{E/\widetilde{Y}}$. Define  $\widetilde{\cD}$ as the subcategory
\begin{equation*}
    \widetilde{\cD}\coloneqq\{\Fca \in \Db(\widetilde{Y})\mid j^*\Fca \in \cB_0\}.
\end{equation*}

\begin{proposition}\label{prop:D-tilde-SOD}{\cite[Proposition~4.1]{Kldcr}} Consider the notation fixed in \eqref{eq:dual Lef}. The pushforward functor $j_*$ is fully faithful on $\cB_i(-i)$ for $1 \le i \le m-1$ and we have a semiorthogonal decomposition
\begin{equation*}
    \Db(\widetilde{Y})=\langle j_* \cB_{m-1}(1-m), \dots, j_* \cB_{1}(-1), \widetilde{\cD} \rangle.
\end{equation*}
\end{proposition}

\begin{theorem}\label{thm:ldcr}{\cite[Theorem~4.4, Proposition~4.5]{Kldcr}}
Consider the notation fixed in \eqref{eq:dual Lef}. Suppose that $\cB_0\subset \Db(E)$ contains $\rho^*(\Dperf(Z))$. Then the functor $\sigma^*$ factors through $\widetilde{\cD}$ and $(\widetilde{\cD}, \sigma_*, \sigma^*)$ is a categorical resolution of $\Db(Y)$ where 
\begin{equation*}
    \sigma_*\colon \widetilde{\cD}\rightarrow \Db(Y), \quad\quad \sigma^*\colon \Dperf(Y) \rightarrow \widetilde{\cD}.
\end{equation*}
If in addition $Y$ is Gorenstein, and $\Canonical_{\widetilde{Y}}=\sigma^* \Canonical_Y\otimes\cO((m-1)E)$, and $\rho^*(\Dperf(Z))\subset \cB_{m-1}$, then the categorical resolution $(\widetilde{\cD}, \sigma_*, \sigma^*)$ is weakly crepant.
\end{theorem}

\subsection{Localization functors and their kernels}

In this section we review results from \cite{KLcr, Ehf} which will allow us to compute the kernels of certain categorical resolutions.

\begin{definition}\label{def:localiz-functor}
Let $\cT$ and $\cT'$ be triangulated categories.
\begin{enumerate}
    \item An exact functor $F\colon \cT\rightarrow \cT'$ is a \emph{localization} if the induced functor $\overline{F}\colon \cT/\ker(F)\rightarrow \cT'$ is an equivalence. 
    \item An exact functor $F\colon \cT\rightarrow \cT'$ is a \emph{localization up to direct summands} if $F\colon\cT\to\im(F)$ is a localization onto a dense subcategory of $\cT'$, in symbols $\im(F)^\oplus=\cT'$.\footnotemark
\end{enumerate}
\footnotetext{The terminology ``categorical contraction'' is preferred for this notion in \cite[Definition~1.10]{KuzSh}.}
\end{definition}

\begin{definition}[Nonrational locus]{\cite[Definition~6.1]{KLcr}}
Let $\sigma\colon X\rightarrow Y$ be a proper birational morphism. A closed subscheme $Z\subset Y$ is called a \emph{nonrational locus} of $Y$ with respect to $\sigma$ if the natural morphism
\[
  \Ideal_Z\rightarrow \sigma_*\Ideal_{\sigma^{-1}(Z)}
\]
is an isomorphism in $\Db(Y)$.
Here $\Ideal_Z\subset\cO_Y$ denotes the ideal sheaf of $Z\subset Y$, and $\sigma^{-1}(Z)$ is the scheme-theoretic pre-image of $Z$, so that $\Ideal_{\sigma^{-1}(Z)}=\sigma^{-1}\Ideal_Z\cdot\cO_X$. 
\end{definition}

\begin{theorem}{\cite[Theorem~8.22]{Ehf}} \label{thm:efimov}
Let $\sigma\colon X\rightarrow Y$ be a proper morphism such that $\sigma_*\cO_X=\cO_Y$. Assume that there is a subscheme $Z\subset Y$, such that all its infinitesimal neighborhoods $Z_k$, for $k\geq 1$, are nonrational loci of $Y$ with respect to $\sigma$. Consider the cartesian diagram
\[
\begin{tikzcd}
E \ar[d, "\rho"] \ar[r, "j"] & X \ar[d, "\sigma"'] \\
Z \ar[r, ""] & Y.
\end{tikzcd}
\]
Assume that the functor $\rho_*\colon \Db(E)\rightarrow \Db(Z)$ is a localization up to direct summands. If $\sigma$ is an isomorphism outside $Z$, then $\sigma_*\colon \Db(X)\rightarrow \Db(Y)$ is a localization up to direct summands with kernel classically generated by $j_*( \ker(\rho_*))$.
\end{theorem}

We verify the hypotheses of \cref{thm:efimov} for blow-ups of certain affine cones. The following corollary is remarked in passing after \cite[Theorem~1.10]{Ehf}; we provide a proof for the sake of completeness.

\begin{corollary}\label{Cor:ker}
Let $Y\subset \AA^{n+1}$ be the cone over a projectively normal smooth Fano variety $W \subset \PP^n$. Let $Z=\{0\}$ be the singular point of $Y$. Let $\sigma\colon \tY\rightarrow Y$ be the blow-up at the singular point $Z$ and $E=W$ its exceptional divisor. Then, $\sigma_*\colon \Db(\tY)\rightarrow \Db(Y)$ is a localization up to direct summands with kernel classically generated by $j_*(\langle \cO_E \rangle^\perp)$, where the orthogonal $\langle \cO_E \rangle^\perp$ is taken in $\Db(E)$.
\end{corollary} 

\begin{proof}
We verify that the hypotheses of \cref{thm:efimov} hold. First note that $\sigma\colon \tY\rightarrow Y$ is a resolution of singularities for $Y$; in particular, it is an isomorphism outside $Z$. Moreover, the exceptional locus $E$ is isomorphic to the Fano variety $W$. As $Y$ is an affine cone over $W$, its coordinate ring is isomorphic to the homogeneous coordinate ring of $W$, which is integrally closed as $W$ is projectively normal, hence $Y$ is normal.

Recall that a cone over a Fano variety has rational singularities by \cite[Corollary~3.4]{kollar2013singularities}, hence, $\sigma_* \cO_{\tY}=\cO_Y$. Let $\rho\colon E\rightarrow Z$ be the restriction of $\sigma$ to $E$. As $E$ is a Fano variety, we have that $\cO_E$ is exceptional by Kodaira's vanishing theorem. As a consequence, we have $\rho_* \cO_E=\Ho^\bullet(E,\cO_E)=\mathbb{C}=\cO_Z$. 
We now prove that  $\rho_*$ is a localization. Since the functor $\rho_*$ has a left adjoint $\rho^*$, by \cite[Remark~3.3]{Ehf} it is a localization if and only if $\rho^*$ is fully faithful. This is indeed the case by the projection formula applied to $\rho_*$ using $\rho_* \cO_E=\cO_Z$ (see \cite[Lemma~2.4]{Kldcr} for details). Finally, considering the decomposition induced on $\Db(E)= \langle \ker(\rho_*), \rho^*\Db(Z) \rangle$, \cf \cite[Lemma~2.3]{Kuzview}, we have
\[
  \ker(\rho_*)=(\rho^*\Db(Z))^\perp=\langle \cO_E \rangle^\perp.
\]
The last thing to check in order to apply \cref{thm:efimov} is that the canonical map
\[
  \Ideal_Z^k\rightarrow \sigma_*(\sigma^{-1}(\Ideal_Z^k)\cdot \cO_{\tY})=\sigma_*\Ideal_{\sigma^{-1}(Z_k)}
\]
is an isomorphism for $k\geq 1$, where $Z_k$ is the $k$-th formal neighbourhood of $Z$. By the construction of blow-ups, the variety $\tY$ is defined as $\Proj(\bigoplus_{i=0}^\infty \Ideal_Z^i)$. On the other hand, the graded sheaf of modules corresponding to $\sigma^{-1}(\Ideal_Z^k)\cdot \cO_{\tY}$ is $\bigoplus_{i=0}^\infty \Ideal_Z^{k+i}$, which is equal to $\cO_{\tY/Y}(k)$, where $\cO_{\tY/Y}(1)$ is the twisting sheaf on the blow-up $\tY$. We recall that $\cO_{\tY/Y}(1)= \cO_{\tY}(-E)$ and $\cO_E(1)=\cO_E(-E)$. 
Consider for $k\geq 0$ the short exact sequence of sheaves on $Y$
\begin{equation}\label{eq:In}
    0\rightarrow  \Ideal^{k+1}_Z\rightarrow \Ideal^{k}_Z \rightarrow \Ideal^k_Z/\Ideal^{k+1}_Z  \rightarrow 0,
\end{equation}
and the short exact sequence of sheaves on $\tY$
\begin{equation}\label{eq:X-seq}
    0\rightarrow \cO_{\tY}(-(k+1)E)\rightarrow \cO_{\tY}(-kE)\rightarrow \cO_E(-kE) \rightarrow 0,
\end{equation}
as well as the morphism of exact triangles
\begin{equation}\label{eq:diagramnew}
    \begin{tikzcd}
    \Ideal_Z^{k+1}\ar[d] \ar[r] &\Ideal_Z^{k} \ar[r]\ar[d] &\Ideal^k_Z/\Ideal^{k+1}_Z  \ar[d,dotted]\\
    \sigma_*\cO_{\tY}(-(k+1)E) \ar[r] &\sigma_*\cO_{\tY}(-kE)  \ar[r] & \sigma_*\cO_E(-kE),
    \end{tikzcd}
\end{equation}
where the upper row is the triangle \eqref{eq:In} and the lower row comes from the application of $\sigma_*$ to \eqref{eq:X-seq}. We claim that the induced map $\Ideal^k_Z/\Ideal^{k+1}_Z  \rightarrow \sigma_* \cO_E(-kE)$ is an isomorphism for $k\geq 0$. As $Z$ is a point, it is enough to study the stalk of the morphism at $Z$.
Let $\operatorname{R}(E)$ be the homogeneous coordinate ring of $E=W\subset\mathbb{P}^n$. By definition, the affine coordinate ring of $Y$, namely $\operatorname{K}[Y]$, is just $\operatorname{R}(E)$ without its grading.
Identifying $\Ideal_Z\subset \operatorname{K}[Y]$ with $(x_0,\dots, x_n)$, we obtain that $\Ideal_Z^{k}/\Ideal^{k+1}_Z$ corresponds to the space of homogeneous polynomials of degree $k$ in $\operatorname{K}[Y]$. On the other hand, by Kodaira vanishing, we have that $\mathrm{H}^i(E,\cO_E(k))=0$ for any $i>0$, so we obtain $\sigma_*  \cO_E(-kE) = \mathrm{H}^0(E,\cO_E(k))$, which is isomorphic to the space of homogeneous polynomials of degree $k$ in $\operatorname{R}(E)$. By projective normality, we have that the composition $\mathrm{H}^0(\mathbb{P}^n,\cO_{\mathbb{P}^n}(k))\rightarrow \Ideal_Z^{k}/\Ideal^{k+1}_Z\rightarrow \mathrm{H}^0(E,\cO_E(k))$ is surjective, \cf \cite[Exercise~II.5.14(d)]{Rag}, hence the map  $\Ideal_Z^{k}/\Ideal^{k+1}_Z\rightarrow \mathrm{H}^0(E,\cO_E(k))$ is surjective. As both source and target of the latter are vector spaces of the same dimension, the map is an isomorphism.

To conclude the proof, we prove inductively that the canonical maps $\Ideal_Z^{k}\rightarrow \sigma_*\cO_{\tY}(-kE)$ are isomorphisms. The base case $k=0$ of the induction is given by the isomorphism $\sigma_*\cO_{\tY}=\cO_Y$. Then by the induction hypothesis the map $\Ideal^k_Z\rightarrow \sigma_*\cO_{\tY}(-kE)$ is an isomorphism, hence the canonical morphism on the left in \eqref{eq:diagramnew} is an isomorphism, concluding the inductive step. As we showed that $Z_k$ is a nonrational locus of $Y$ for $k\geq 1$, we can apply \cref{thm:efimov} and obtain the statement.
\end{proof}

\begin{remark} \label{rmk:nodalcase}
Note that \cref{Cor:ker} remains valid for varieties $Y$ with an isolated singular point $y$ which look, upon restriction to a formal neighborhood of $y$ in $Y$, like the cone singularity in the corollary.
Indeed, the crucial part of the proof is the check that the infinitesimal neighborhoods of the singularity are nonrational loci.
Now use that $\operatorname{Spec}(\widehat{\cO_{Y,y}})\to\operatorname{Spec}(\cO_{Y,y})$ is faithfully-flat, \cf \cite[{Tag~00MC}]{stacks-project}, and $\operatorname{Spec}(\cO_{Y,y})\to Y$ is flat, so the nonrational locus condition can be checked after base-change to $\operatorname{Spec}(\widehat{\cO_{Y,y}})$.
\end{remark}

\subsection{Derived base-change}
The last ingredient we need in the derived categories setting is the following base-change result. 

\begin{proposition}\label{prop:com}
Consider a cartesian square of varieties
\[
\begin{tikzcd}
X\times_S Y\ar[r, "q"] \ar[d, "p"] & Y\ar[d, "g"'] \\
X\ar[r, "f"] & S.
\end{tikzcd}
\]
Suppose that $g$ is a closed immersion and local complete intersection morphism, $X$ is Cohen--Macaulay, and $\codim_{X}(X\times_S Y)=\codim_S(Y)$. Then 
\[
  q_*p^*=g^*f_*.
\]
\end{proposition}

\begin{proof}
The proposition is a corollary of Tor-independent base-change, \cf \cite[Tag~08IB]{stacks-project}.
In slightly more detail: Since local complete intersection immersions are Koszul-regular immersions, \cf\cite[Tag~09CC]{stacks-project}, one can use the Koszul complex to compute higher Tor groups.
The regular sequences that define $Y\subset S$ locally stay regular on $X$ because of the codimension assumption and the unmixedness theorem, \cf\cite[Tag~02JN]{stacks-project}.
So the Koszul complex stays exact after tensoring with $\cO_X$, and we see that higher Tor groups vanish, as required to apply Tor-independent base-change.

A proof can also be found in \cite[Corollary~2.27]{Ket}.
\end{proof}

\subsection{Spinor bundles on quadric hypersurfaces}

In this subsection we summarize some properties of spinor bundles on quadric hypersurfaces. Let $Q\subset\mathbb{P}^{n+1}$ be the (unique up to isomorphism) smooth quadric hypersurface of dimension $n$. The definition of spinor bundles on $Q$, given in \cite{Oq}, depends on the parity of $n$. 

Assume first that $n=2m+1$ is odd; in this case, the maximal dimension of a (projective) linear subspace contained in $Q$ is $m$. The parameter space for the $m$-planes contained in $Q$ is an irreducible smooth projective variety $S$. Let $\Oca_S(1)$ be the ample generator of $\mathrm{Pic}(S) \simeq \mathbb{Z}$; it can be shown that $\dim \Ho^0(S, \Oca_S(1)) = 2^{m+1}$. Now, for any $x \in Q$, consider the embedding
\[
i_x \colon S_x \coloneqq \{\PP^m \subset Q \mid x\in \PP^m \} \to S = \{ \PP^m \subset Q \}.
\]
The induced restriction map $\Ho^0(S, \Oca_S(1)) \to \Ho^0(S_x, i_x^*\Oca_S(1))$ turns out to be surjective, so its dual yields an inclusion
\[
\Ho^0(S_x, i_x^*\Oca_S(1))^\vee \to \Ho^0(S, \Oca_S(1))^\vee.
\]
Since $\dim \Ho^0(S_x, i_x^*\Oca_S(1)) = 2^m$ for any $x \in Q$, we obtain a morphism
\[
s \colon Q \to \mathrm{Gr}(2^m, 2^{m+1}).
\]
The spinor bundle $\Sca$ on $Q$ is defined as the pullback by $s$ of the tautological subbundle on $\mathrm{Gr}(2^m, 2^{m+1})$.

Let us move on to the case of a quadric of even dimension $n=2m$. The maximal dimension of a linear subspace contained in $Q$ is $m$. The parameter space for the $m$-planes contained in $Q$ has two connected components $S'$ and $S''$. Both $S'$ and $S''$ are smooth irreducible projective varieties. Let $\Oca_{S'}(1)$ and $\Oca_{S''}(1)$ be the ample generators of $\mathrm{Pic}(S') \simeq \mathbb{Z}$ and $\mathrm{Pic}(S'') \simeq \mathbb{Z}$, respectively; it can be shown that $\dim \Ho^0(S', \Oca_
{S'}(1)) = \dim \Ho^0(S'', \Oca_
{S''}(1)) = 2^m$. Now, for any $x \in Q$, consider the embeddings
\[
i'_x \colon S'_x = \{ \PP^m \in S' \mid x \in \PP^m \} \to S' \ \ \ \text{and} \ \ \ i''_x \colon S''_x = \{ \PP^m \in S'' \mid x \in \PP^m \} \to S''.
\]
The induced restriction maps $\Ho^0(S', \Oca_{S'}(1)) \to \Ho^0(S'_x, (i'_x)^*\Oca_{S'}(1))$ and $\Ho^0(S'', \Oca_{S''}(1)) \to \Ho^0(S''_x, (i''_x)^*\Oca_{S''}(1))$ turn out to be surjective. By passing to the dual we obtain the inclusions
\[
\Ho^0(S'_x, (i'_x)^*\Oca_{S'}(1))^\vee \to \Ho^0(S', \Oca_{S'}(1))^\vee \ \ \ \text{and} \ \ \ \Ho^0(S''_x, (i''_x)^*\Oca_{S''}(1))^\vee \to \Ho^0(S'', \Oca_{S''}(1))^\vee.
\]
Since $\dim \Ho^0(S'_x, (i'_x)^*\Oca_{S'}(1)) = \dim \Ho^0(S''_x, (i''_x)^*\Oca_{S''}(1)) = 2^{m-1}$ for any $x \in Q$, we obtain two morphisms
\[
s' \colon Q \to \mathrm{Gr}(2^{m-1}, 2^m) \ \ \ \text{and} \ \ \ s'' \colon Q \to \mathrm{Gr}(2^{m-1}, 2^m).
\]
The spinor bundle $\Sca'$ (resp.\ $\Sca''$) on $Q$ is defined as the pullback by $s'$ (resp.\ $s''$) of the tautological subbundle on $\mathrm{Gr}(2^{m-1}, 2^m)$.
We write $\cS$, respectively $\cS'$, $\cS''$, for the spinor bundle(s) on the odd, respectively even, dimensional quadric $Q$. These bundles enjoy the following properties.
\begin{theorem}\label{thm:qua}~
\begin{enumerate}
    \item The spinor bundles are stable, \cf\cite[Theorem~2.1]{Oq}.
    \item Suppose $Q$ is an even dimensional quadric and let $\cS'$, $\cS''$ be the two spinor bundles. Let $i\colon Q'\rightarrow Q$ be the closed immersion of a smooth hyperplane section, with spinor bundle $\cS$. Then $i^*\cS'=i^*\cS''=\cS$, \cf\cite[Theorem~1.4(i)]{Oq}.
    \item If $\cS$ is either the spinor bundle on the odd dimensional quadric or any of the two spinor bundles on the even dimensional quadric, then $\Ho^i(Q, \cS(k))=0$ for $0 <i<n$ and arbitrary $k\in\mathbb{Z}$.
    Furthermore $\Ho^0(Q,\cS(k))=0$ for $k\leq0$, and $\dim \Ho^0(Q, \cS(1))=2^{[(n+1)/2]}$, where $n$ is the dimension of $Q$, \cf\cite[Theorem~2.3]{Oq}.
    \item Suppose the quadric $Q$ has odd dimension $n=2m+1$. We have a short exact sequence
    \begin{equation}\label{eq:taut-odd}0\rightarrow\cS\rightarrow \cO_{Q}^{\oplus 2^{m+1}}\rightarrow\cS(1)\rightarrow0,
    \end{equation}
    and $\cS^{\vee}=\cS(1)$, \cf\cite[Theorem~2.8(i)]{Oq}.
    \item Suppose the quadric $Q$ has even dimension $n=2m$. We have short exact sequences
    \begin{equation}\label{eq:tautological-odd-spinor}
        \begin{split}
        &0\rightarrow \cS'\rightarrow \cO_{Q}^{\oplus 2^m}\rightarrow \cS''(1)\rightarrow 0,\\
        &0\rightarrow \cS''\rightarrow \cO_{Q}^{\oplus 2^m}\rightarrow \cS'(1)\rightarrow 0.
    \end{split}
    \end{equation}
    Furthermore, if $n\equiv0 \pmod 4$, then $\cS'^{\vee}=\cS'(1)$ and $\cS''^{\vee}=\cS''(1)$, and if $n\equiv2 \pmod 4$, then $\cS'^{\vee}=\cS''(1)$ and $\cS''^{\vee}=\cS'(1)$, \cf\cite[Theorem~2.8(ii)]{Oq}.
    \item Spinor bundles are exceptional. If $Q$ is even dimensional, $\cS'$ and $\cS''$ are orthogonal to each other, \cf\cite{Ksd}.
\end{enumerate}
\end{theorem}

We summarize here some cohomology computations. 

\begin{lemma}\label{lem:canonical bundle quadric}
Let $Q\subset\mathbb{P}^{n+1}$ be the smooth quadric of dimension $n$. Then
\begin{equation}\label{eq:canonical}\Canonical_Q=\cO_Q(-n)\end{equation} and the following cohomology groups vanish:
\begin{equation}\label{eq:O(-t)-vanish}
\Ho^\bullet(Q, \cO_{Q}(-k))=0 \quad \text{for} \quad k=1,\dots,n-1.
\end{equation}
\end{lemma}
\begin{proof}
As $Q$ is a smooth hypersurface of degree $2$ in $\PP^{n+1}$, by the adjunction formula we have $\Canonical_Q=\cO_Q(-n-2+2)=\cO_Q(-n)$. The vanishing statement \eqref{eq:O(-t)-vanish} follows from Kodaira's vanishing theorem.
\end{proof}

\begin{remark}\label{rem:twist-vanish} Let $\cS$ be any spinor bundle on the smooth quadric $Q$ of dimension $n$. Using Serre duality and \cref{thm:qua}(c)-(e), we have $\Ho^n(Q,\cS(k))=0$ for $k \ge 1-n$. In particular, $\Ho^\bullet(Q, \cS(k))=0$ for $1-n \le k \le 0$.
\end{remark}

\begin{lemma}\label{lem:odd-quad}
Let $Q\subset\mathbb{P}^{n+1}$ be the smooth quadric of odd dimension $n=2m+1$. We have
\begin{equation}\label{eq:S(t)-vanish}
    \Hom^\bullet(\cS(k),\cS)=\begin{cases} \mathbb{C} &\text{if}\ k = 0 \\
    \mathbb{C}[-1] &\text{if}\ k = 1.
    \end{cases}
\end{equation}
\end{lemma}
\begin{proof}
The isomorphism for $k=0$ follows from the exceptionality of $\cS$, see \cref{thm:qua}(f).
For the proof of the second isomorphism, we use sequence \eqref{eq:taut-odd}. Consider the long exact sequence induced by applying $\Hom^\bullet(-,\cS)$.
This provides the exact triangle
\begin{equation*}
    \Hom^\bullet(\cS,\cS)\leftarrow  \Hom^\bullet(\cO^{\oplus 2^{m+1}}_Q,\cS) \leftarrow \Hom^\bullet(\cS(1),\cS).
\end{equation*}
As the central term vanishes by \cref{rem:twist-vanish}, we obtain
\begin{equation*}
    \Hom^\bullet(\cS(1),\cS)=\Hom^\bullet(\cS,\cS)[-1]=\mathbb{C}[-1].
    \qedhere
\end{equation*}
\end{proof}

\begin{lemma}\label{lem:even-quad}
Let $Q\subset\mathbb{P}^{n+1}$ be the smooth quadric of even dimension $n=2m\geq 2$. Let $\cS'$ and $\cS''$ be its spinor bundles. We have
\begin{equation}\label{eq:S'-vanish}
    \Hom^\bullet(\cS'(k),\cS')=\Hom^\bullet(\cS''(k),\cS'')=\begin{cases} \mathbb{C} & \text{if}\ k = 0   \\
    0 &\text{if}\ k = 1,
    \end{cases}
\end{equation}
and 
\begin{equation}\label{eq:S''-vanish}
    \Hom^\bullet(\cS''(k),\cS')=\Hom^\bullet(\cS'(k),\cS'')= \begin{cases}
    0 & \text{if}\ k = 0\\
    \mathbb{C}[-1] &\text{if}\ k = 1.
    \end{cases}
\end{equation}
\end{lemma}

\begin{proof}

If $k=0$, the isomorphism~\eqref{eq:S'-vanish} holds because $\cS'$ is exceptional by \cref{thm:qua}(f). To prove the vanishing of $\Hom^\bullet(\cS'(1),\cS')$, consider the defining sequence of a smooth hyperplane section $i\colon Q'\rightarrow Q$ tensored with $\cS'$
\[
  0\rightarrow \cS'(-1)\rightarrow \cS'\rightarrow i_*i^*\cS'\rightarrow0.
\]
Applying $\Hom^\bullet(\cS',-)$ and using adjunction we get 
\begin{equation*}
    \Hom^\bullet(\cS',\cS'(-1))\rightarrow  \Hom^\bullet(\cS',\cS') \rightarrow \Hom^\bullet(\cS',i_*i^*\cS')=\Hom^\bullet(i^*\cS',i^*\cS').
\end{equation*}
Recall that by \cref{thm:qua}(b) we have $i^*\cS'=\cS$, where $\cS$ is the spinor bundle on $Q'$. As spinor bundles are exceptional, we have $\Hom^\bullet(\cS',\cS')=\mathbb{C}=\Hom^\bullet(\cS,\cS)$. Moreover, the map $\Hom^0(\cS',\cS') \rightarrow \Hom^0(\cS',i_*i^*\cS')$ is injective, hence an isomorphism. We conclude that $\Hom^\bullet(\cS'(1),\cS')=\Hom^\bullet(\cS',\cS'(-1))=0$.

We proceed with the proof of \eqref{eq:S''-vanish}. The vanishing for $k=0$ holds by \cref{thm:qua}(f). We calculate $\Hom^\bullet(\cS''(1),\cS')$. Applying $\Hom^\bullet(-,\cS')$ to the sequence \eqref{eq:tautological-odd-spinor}
\[
  0\rightarrow \cS'\rightarrow \cO^{\oplus 2^m}_Q\rightarrow \cS''(1)\rightarrow0,
\]
we get 
\begin{equation*}
    \Hom^\bullet(\cS',\cS')\leftarrow  \Hom^\bullet(\cO^{\oplus 2^m}_Q,\cS') \leftarrow \Hom^\bullet(\cS''(1),\cS').
\end{equation*}
As the central term vanishes by \cref{rem:twist-vanish}, we obtain
\begin{equation*}
    \Hom^\bullet(\cS''(1),\cS')
    =\Hom^\bullet(\cS',\cS')[-1]=\mathbb{C}[-1].
    \qedhere
\end{equation*}
\end{proof}

We end this section by recalling Kapranov's Lefschetz decomposition for quadrics. 

\begin{theorem}\cite[Lemma~2.4]{KPhpdq}\label{thm:qsod}
Let $Q\subset \mathbb{P}^{n+1}$ be the smooth quadric of dimension $n$. Then we have the dual Lefschetz decomposition
\begin{equation}\label{eq:Lefschetz-quadric}
    \Db(Q)=\langle\cB_{n-1}(1-n) , \dots, \cB_1(-1), \cB_0 \rangle.
\end{equation}
Here, if $n$ is odd, we have
\begin{equation*}
    \cB_0 = \langle \cS, \cO_{Q} \rangle \quad \mathrm{and} \quad \cB_1=\dots=\cB_{n-1}=\langle \cO_{Q}\rangle,
\end{equation*}
where the bundle $\cS$ is the unique spinor bundle on $Q$. If $n$ is even, we have
\begin{equation*}
    \cB_0=\cB_1 = \langle \cS', \cO_{Q} \rangle \quad \mathrm{and} \quad \cB_2 =\dots =\cB_{n-1}=\langle \cO_{Q}\rangle,
\end{equation*}
where $\cS'$ is any of the two spinor bundles on $Q$.
\end{theorem}
\begin{proof}
For an odd dimensional quadric we have by \cite{Ksd} the semiorthogonal decomposition
\[
  \Db(Q)=\langle \cS,\cO_{Q},\cO_{Q}(1),\dots,\cO_{Q}(n-1)\rangle.
\]
It suffices to suitably group its components and apply \cref{lemma:serre1}(a) to get the desired dual Lefschetz decomposition.

For a quadric of dimension $n=2m$ we have by \cite{Ksd} the semiorthogonal decomposition
\[
  \Db(Q)=\langle \cS',\cS'',\cO_{Q},\cO_{Q}(1),\dots,\cO_{Q}(n-1)\rangle.
\]
We claim that $\mathbb{R}_{\cO_{Q}}\cS''=\cS'(1)[-1]$. 
First, we have
\[
  \Hom^\bullet(\cS'',\cO_{Q})=\Hom^\bullet(\cO_{Q},\cS''^{\vee})= \Ho^\bullet(Q,\cS'''(1))
\]
where $\cS'''$ is one of the spinor bundles depending on the parity of $m$, see \cref{thm:qua}(e). By \cref{thm:qua}(c), we have the isomorphism
\[
  \Ho^\bullet(Q,\cS'''(1))=\mathbb{C}^{\oplus 2^m}.
\]
Then, using the exact sequence \eqref{eq:tautological-odd-spinor}, we obtain that $\cone(\cS''\rightarrow \mathbb{C}^{\oplus 2^m}\otimes \cO_{Q})=\cS'(1)$, which shows that $\mathbb{R}_{\cO_{Q}}\cS''=\cS'(1)[-1]$.
By \cref{SOD-mutations} we deduce the semiorthogonal decomposition
\[
  \Db(Q)=\langle \cS',\cO_{Q}, \cS'(1), \cO_{Q}(1),\dots,\cO_{Q}(n-1)\rangle.
\]
Tensoring by $\cO_{Q}(-1)$ and applying \cref{lemma:serre1}(a) as before, we get the desired dual Lefschetz decomposition.
\end{proof}

%------------------------------------------------------------------------------%
% Section: General Case
%------------------------------------------------------------------------------%

\section{Categorical resolutions of nodal varieties} \label{sec:generalcase}

In this section we prove \cref{thm:ns}, which is obtained from \cref{prop:catresker,prop:sperical-odd,prop:seven}.
Let $Y$ be a quasiprojective variety with an isolated nodal singularity $y$. Let $\sigma\colon \widetilde{Y}\rightarrow Y$ be the resolution of singularities provided by the blow-up at the singular point. Recall that the exceptional divisor $j\colon Q\rightarrow \widetilde{Y}$ is isomorphic to the smooth quadric of dimension $\dim(Y) -1$. Let $\cS$ be the spinor bundle on $Q$ if $\dim(Y)$ is even, and denote by $\cS',\cS''$ the spinor bundles if $\dim(Y)$ is odd. 
Recall from \cref{subsec:nodalsing} that $\cO_Q(1)= j^*\cO_{\widetilde{Y}}(-Q)$.

For the sake of simplicity, let us assume first that $Y$ is projective: we shall explain how to adjust the proofs when $Y$ is quasiprojective in \cref{rem:quasiprojective}. We start with some observations on the properties of certain objects in $\Db(\tY)$. 

\begin{lemma}\label{lem:excep}
If $\dim(Y) \geq 3$, then $j_*\cO_Q(k)$ is exceptional. Moreover, if $\dim(Y)$ is odd, then $j_*\cS'$ and $j_*\cS''$ are exceptional as well, and we have
\begin{equation}\label{eq:push-S''-S'}
    \Hom^\bullet(j_*\cS',j_*\cS'')= \bC[-2].
\end{equation} 
If $\dim(Y)$ is even, then we have that 
\begin{equation}\label{eq:push-hom(S-S)}
    \Hom^\bullet(j_*\cS,j_*\cS)=\bC \oplus \bC[-2].
\end{equation}
\end{lemma}
\begin{proof}
By \cref{prop:D-tilde-SOD}, the functor $j_*$ is fully faithful on the subcategory generated by the exceptional object $\cO_Q$. It follows that $j_*\cO_Q(k)=j_*\cO_Q \otimes \cO_{\tY}(-kQ)$ is exceptional too. 

Now let us assume that $\dim(Y)$ is odd. Note that the role of the spinor bundles $\cS'$ and $\cS''$ is interchangeable. Applying \cref{prop:D-tilde-SOD} to the Lefschetz decomposition \eqref{eq:Lefschetz-quadric}, we get that $j_*\cS'$ and $j_*\cS''$ are exceptional.
Next, we compute $\Hom^\bullet(j_*\cS',j_*\cS'')$, which is isomorphic to $\Hom^\bullet(j^* j_*\cS',\cS'')$ by adjunction. 
Consider the exact triangle on $Q$
\begin{equation}\label{eq:restrict-sequence}
  j^*j_*\cS'\rightarrow \cS'\rightarrow \cS'(-Q)[2]=\cS'(1)[2],
\end{equation}
and the associated long exact sequence obtained by applying $\Hom^\bullet(-, \cS'')$. By \cref{thm:qua}(f) and \cref{lem:even-quad} we know 
\begin{equation*}
    \Hom^\bullet(\cS',\cS'')=0, \quad \quad \Hom^\bullet(\cS'(1),\cS'')=\mathbb{C}[-1].
\end{equation*}   
Substituting these equalities, we obtain
\[
  \Hom^\bullet(j_*\cS',j_*\cS'')= \Hom^\bullet(\cS'(1),\cS'')[-1]= \mathbb{C}[-2],
\]
proving the equality \eqref{eq:push-S''-S'}.  

Following the same strategy, we compute $\Hom^\bullet(j_* \cS, j_* \cS)$ when $\dim(Y)$ is even. 
By applying $\Hom^\bullet(-, \cS)$ to the exact triangle \eqref{eq:restrict-sequence}, we obtain
\[
  \Hom^\bullet(j^*j_*\cS,\cS)\leftarrow \Hom^\bullet(\cS,\cS)\leftarrow  \Hom^\bullet(\cS(1)[2],\cS).
\]
Recalling that
\[
  \Hom^\bullet(\cS,\cS)=\mathbb{C}, \quad \Hom^\bullet(\cS(1),\cS)=\mathbb{C}[-1],
\]
by \cref{lem:odd-quad}, we obtain 
$\Hom^i(j^*j_*\cS,\cS)=0$ except for $i=0, 2$, for which it is equal to $\mathbb{C}$.
\end{proof}

\begin{remark}
  Note that by \cref{lem:excep} the objects $j_* \cO_Q(k)$ are exceptional, so the mutation functor $\mathbb{R}_{j_* \cO_Q(k)}$ is well defined. The same remark holds for $j_*\cS''$ when $\dim(Y)$ is odd.
\end{remark}

\begin{lemma}\label{lem:R-mutate-odd}
If $\dim(Y)$ is even, we have the isomorphisms
\begin{equation*}
\mathbb{R}_{j_* \cO_Q(k)} (j_* \cS) =j_* \cS
\end{equation*}
for $2-\dim(Y) \le k \le -1$. Moreover, for all $k \in \mathbb{Z}$, we have
\begin{equation*}
\mathbb{R}_{j_* \cO_Q(k)} (j_* \cS(k)) =j_* \cS(k+1)[-1].
\end{equation*}
\end{lemma}
\begin{proof}
Again, we use the exact triangle on $Q$
\begin{equation}\label{eq:restrict-sequence1}
  j^*j_*\cS\rightarrow \cS\rightarrow \cS(1)[2].
\end{equation}
We prove the first isomorphism. By adjunction, $\Hom^\bullet(j_*\cS,j_*\cO_Q(k))=\Hom^\bullet(j^*j_*\cS,\cO_Q(k))$.
Applying $\Hom^\bullet(-,\cO_Q(k))$ to \eqref{eq:restrict-sequence1} we get the exact triangle
\[
  \Hom^\bullet(j^*j_*\cS,\cO_Q(k))\leftarrow \Hom^\bullet(\cS,\cO_Q(k))\leftarrow  \Hom^\bullet(\cS(1)[2],\cO_Q(k)).
\]
Now we have by \cref{thm:qua}(d) that
\begin{align*}
\Hom^\bullet(\cS,\cO_Q(k))&=\Ho^\bullet(Q,\cS^{\vee}(k))=\Ho^\bullet(Q,\cS(k+1))\quad \text{and}\\ \Hom^\bullet(\cS(1),\cO_Q(k))&=\Ho^\bullet(Q,\cS^{\vee}(k-1))=\Ho^\bullet(Q,\cS(k)).
\end{align*}
Since $2-\dim(Y) \le k \le -1$, both these terms vanish by \cref{rem:twist-vanish}, thus
\begin{equation}\label{eq:push-pull-vanish}
    \Hom^\bullet(j^*j_*\cS,\cO_Q(k))=0 \quad \textrm{for} \quad 2-\dim(Y) \le k \le -1.
\end{equation}
 We conclude that
\[
\mathbb{R}_{j_* \cO_Q(k)} (j_* \cS) =j_* \cS.
\]

We prove now the second isomorphism. Let $2m+1=\dim(Y)-1$. 
Twisting the exact sequence \eqref{eq:taut-odd} by $\cO_Q(k)$ and taking the pushforward along $j$, we get the exact triangle
\begin{equation}\label{eq:mutation-triangle-new}
    j_*\cS(k+1)[-1]\rightarrow j_*\cS(k)\rightarrow j_*\cO_{Q}(k)^{\oplus 2^{m+1}}.
\end{equation}
Since $\Hom^\bullet(j_*\cS(k+1),j_*\cO_{Q}(k))=\Hom^\bullet(j^*j_*\cS(1),\cO_Q)=0$ by \eqref{eq:push-pull-vanish}, this is a mutation triangle. 
This immediately implies the statement.
\end{proof}

\begin{lemma}\label{lem:R-mutate-even}
If $\dim(Y)$ is odd, we have the isomorphisms
\begin{equation*}
\mathbb{R}_{j_* \cO_Q(k)} (j_* \cS') =j_* \cS'
\end{equation*}
for $2-\dim(Y) \le k \le -1$. Moreover, for all $k \in \mathbb{Z}$, we have
\begin{align*}
    \mathbb{R}_{j_* \cO_Q(k)} (j_* \cS'(k)) &=j_* \cS''(k+1)[-1],\\
    \mathbb{R}_{j_* \cO_Q(k)} (j_* \cS''(k)) &=j_* \cS'(k+1)[-1].
\end{align*}
Finally, we have that
\begin{equation*}
\mathbb{R}_{j_* \cS''} (j_* \cS') =\cone(j_*\cS'\rightarrow j_*\cS''[2])[-1].
\end{equation*}
\end{lemma}
\begin{proof}
The first three isomorphisms are proven in the same way as the previous lemma. The last one follows immediately from \eqref{eq:push-S''-S'}.
\end{proof}

We can now come to the study of the categorical resolution of $\Db(Y)$, where $Y$ is the nodal variety from the beginning of this section.

\begin{proposition} \label{prop:catres_part0}
With the notation introduced at the beginning of this section and in \cref{thm:qsod}, set 
\[
  \widetilde{\cD}\coloneqq\{\Fca\in \Db(\widetilde{Y}) \mid j^*\Fca\in \cB_0\}. 
\]
Let $\sigma_*\colon \widetilde{\cD}\rightarrow \Db(Y)$ denote the restriction of the pushforward functor.
Then the pullback functor $\sigma^*\colon \Dperf(Y)\rightarrow \Db(\widetilde{Y})$ factors as $\sigma^*\colon  \Dperf(Y)\rightarrow \widetilde{\cD}$, and $(\widetilde{\cD}, \sigma_*, \sigma^*)$ is a weakly crepant categorical resolution of $\Db(Y)$.    
\end{proposition}

\begin{proof}
Set $n\coloneqq\dim(Y)-1$. Recall the dual Lefschetz decomposition of $\Db(Q)$ introduced in \cref{thm:qsod}
\begin{equation*}
    \Db(Q)=\langle \cB_{n - 1}(1- n),\dots ,\cB_1(-1), \cB_0 \rangle,
\end{equation*}
where
\begin{enumerate}
    \item  $\cB_0=\langle \cS, \cO_Q\rangle$ and $\cB_i=\langle \cO_Q\rangle$ for $1 \le i \le n-1$, if $Y$ is even dimensional,
    \item  $\cB_0=\cB_1=\langle \cS', \cO_Q\rangle$ and $\cB_i=\langle \cO_Q\rangle$ for $2 \le i \le n-1$, if $Y$ is odd dimensional.
\end{enumerate}
Denote by $\rho$ the restriction of $\sigma$ to $Q$; the image of $\rho$ consists of the singular point $y$ of $Y$. Since $\Dperf(y)=\langle \cO_{y} \rangle$, we have $\rho^* \Dperf(y)= \langle \rho^* \cO_{y} \rangle= \langle \cO_Q \rangle \subset \cB_0$.
In fact, $\rho^*\Dperf(y)=\langle \cO_Q\rangle \subset \cB_i$ for all $i$. 
Recall that a variety with nodal singularities is Gorenstein, as discussed in \cref{subsec:nodalsing}. We now compute the discrepancy of the exceptional divisor $Q$.
As $\sigma$ is an isomorphism outside of $Q$, we have $ \Canonical_{\widetilde{Y}}= \sigma^* \Canonical_Y \otimes \cO(k Q)$ for some $k \in \mathbb{Z}$. By the adjunction formula and \cref{lem:canonical bundle quadric} we have that
\begin{align*}
    \cO_{Q}(-n)=\Canonical_{Q} = (\Canonical_{\widetilde{Y}}\otimes \cO(Q))\vert_{Q}
    =(\sigma^* \Canonical_Y \otimes \cO((k+1) Q))\vert_{Q}=\cO_{Q}(-k-1).
\end{align*}
As $\Pic(Q)$ is torsion free, \cf\cite[Exercise~II.6.5c]{Rag}, this implies $k=n-1$.
Then the triple $(\widetilde{\cD},\sigma_*,\sigma^*)$ defined in the proposition is a weakly crepant categorical resolution by \cref{thm:ldcr}.    
\end{proof}

Next, we compute the kernel of the categorical resolution from \cref{prop:catres_part0}. For the rest of this section, we will use $\sigma_*$ to denote the pushforward functor $\Db(\tY)\rightarrow \Db(Y)$, and not its restriction to $\widetilde{\cD}$.

\begin{proposition}\label{prop:catresker}
The kernel $\ker(\sigma_*)\cap \widetilde \cD$ of the weakly crepant categorical resolution $\widetilde{\cD}$ of $\Db(Y)$ is classically generated by a single object $\Tca$, where $\Tca=j_* \cS$ if $\dim(Y)$ is even, and $\Tca=\mathbb{R}_{j_* \cS''} (j_* \cS'[1])=\cone(j_*\cS'\rightarrow j_*\cS''[2])$ if $\dim(Y)$ is odd.
\end{proposition}

\begin{proof}
Set $n\coloneqq\dim(Y)-1$. Note that the conditions of \cref{thm:efimov} are satisfied in our situation as explained in \cref{Cor:ker} and \cref{rmk:nodalcase}. This gives that $\sigma_*:\Db(\tY)\rightarrow \Db(Y)$ is a localization functor up to direct summands, and its kernel is classically generated by $\mathscr{K}\coloneqq \langle j_*(\langle \cO_Q \rangle^\perp) \rangle$, that is, $\ker(\sigma_*) = \mathscr{K}^{\oplus}$ is the idempotent completion of $\mathscr{K}$. We now determine $\ker(\sigma_*) \cap \widetilde{\cD}$. 

On the one hand, $\langle \cO_Q \rangle^\perp$ admits by \cref{thm:qsod} a semiorthogonal decomposition of the form
\begin{enumerate}
    \item  $\langle \cO_Q \rangle^\perp=\langle \cO_Q(1-n),\dots, \cO_Q(-1), \cS \rangle$ if $Y$ is even dimensional,
    \item $\langle \cO_Q \rangle^\perp=\langle \cO_Q(1-n),\dots,\cS'(-1), \cO_Q(-1), \cS' \rangle$ if $Y$ is odd dimensional,
\end{enumerate}
thus the pushforwards of the components along $j$ are a set of generators of $\mathscr{K}$. 

On the other hand, the semiorthogonal decomposition
\begin{equation}\label{eq:toeliminateSOD}\Db(\widetilde{Y})=\langle j_* \cB_{n -1}(1- n), \dots, j_* \cB_1(-1), \widetilde{\cD} \rangle
\end{equation}
induced by \eqref{eq:Lefschetz-quadric} and the fully faithfulness of $j_*$ on $\cB_i(-i)$ for $1 \le i \le n-1$ show that
\begin{enumerate}
    \item $\lbrace j_* \cO_{Q}(1- n), \dots, j_* \cO_Q(-1)
    \rbrace$ if $Y$ is even dimensional,
    \item $\lbrace j_* \cO_{Q}(1- n), \dots, j_* \cS'(-1), j_* \cO_Q(-1)
    \rbrace$ if $Y$ is odd dimensional
\end{enumerate}
are full exceptional collections of $\widetilde\cD^\perp$. 

Now, looking at the generators of $\widetilde{\cD}^\perp$ and $\mathscr{K}$, we obtain $\widetilde{\cD}^\perp \subset \mathscr{K} \subset \ker(\sigma_*)$, which implies that
\begin{equation*}
         \ker (\sigma_*)\cap \widetilde{\cD} = \mathbb{R}_{\widetilde{\cD}^\perp} \ker (\sigma_*).
\end{equation*}
We first assume that $Y$ is even dimensional. Notice that all the generators of $\mathscr{K}$ belong to $\widetilde{\cD}^\perp$ except the pushforward of the spinor bundle, so that $\mathbb{R}_{\widetilde\cD^\perp} \mathscr{K} = \langle \mathbb{R}_{\widetilde\cD^\perp}(j_* \cS) \rangle$.
Since $\mathbb{R}_{\widetilde\cD^\perp} (\mathscr{K}^{\oplus })\subset (\mathbb{R}_{\widetilde\cD^\perp} \mathscr{K})^{\oplus}$, we have the inclusions 
\[
\mathbb{R}_{\widetilde\cD^\perp} (j_*\cS) \subset \mathbb{R}_{\widetilde\cD^\perp} (\ker(\sigma_*))\subset \langle \mathbb{R}_{\widetilde\cD^\perp}(j_* \cS) \rangle^{\oplus}.
\]
Now, as both $\ker (\sigma_*)$ and $\widetilde{\cD}$ are idempotent complete, so is their intersection $\mathbb{R}_{\widetilde\cD^\perp} (\ker(\sigma_*))$.
Thus $\mathbb{R}_{\widetilde\cD^\perp} \ker(\sigma_*) = \langle \mathbb{R}_{\widetilde\cD^\perp}(j_* \cS) \rangle^{\oplus}$. The same argument shows that $\mathbb{R}_{\widetilde\cD^\perp} \ker(\sigma_*) = \langle \mathbb{R}_{\widetilde\cD^\perp}(j_* \cS') \rangle^{\oplus}$ when $Y$ is odd dimensional.

To conclude, it suffices to compute the mutations of the spinor bundles through $\widetilde\cD^\perp$.
When~$Y$ is even dimensional, we have by \cref{mutations-SOD} and \cref{lem:R-mutate-odd} 
\[
\mathbb{R}_{\widetilde\cD^\perp}(j_* \cS)=(\mathbb{R}_{j_*\cO_Q(-1)}\circ \dots \circ \mathbb{R}_{j_*\cO_Q(1-n)})(j_*\cS) = j_*\cS.
\]

When $Y$ is odd dimensional, we consider the exceptional collection of $\widetilde\cD^\perp$ obtained by mutating $j_* \cS'(-1)$ through $j_*\cO_Q(-1)$. Since $\mathbb{R}_{j_* \cO_Q(-1)} (j_* \cS'(-1))=j_*\cS''[-1]$ by \cref{lem:R-mutate-even}, we have
\begin{equation}\label{eq:SOD-R-mutated}
\widetilde\cD^\perp = \langle j_* \cO_{Q}(1- n), \dots, j_* \cO_Q(-1), j_*\cS''\rangle.
\end{equation}
Using \cref{mutations-SOD} and \cref{lem:R-mutate-even} we obtain
\[
\mathbb{R}_{\widetilde\cD^\perp}(j_* \cS')=(\mathbb{R}_{j_* \cS''}\circ \mathbb{R}_{j_*\cO_Q(-1)}\circ \dots \circ \mathbb{R}_{j_*\cO_Q(1-n)})(j_*\cS')=\mathbb{R}_{j_* \cS''}(j_* \cS'),
\]
and also $\mathbb{R}_{j_* \cS''}(j_* \cS')=\cone(j_*\cS'\rightarrow j_*\cS''[2])[-1]$.

These computations yield the desired classical generator of $\ker(\sigma_*) \cap \widetilde{\cD}$ in the even and odd dimensional case.
\end{proof}

Now let $\Tca=j_*\cS$ or $\Tca=\cone(j_*\cS'\rightarrow j_*\cS''[2])$, depending on the parity of the dimension of~$Y$.

\begin{proposition}\label{prop:sperical-odd}
If $Y$ is even dimensional, 
then $\Tca$ is a $2$-spherical object in $\widetilde{\cD}$.
\end{proposition}

\begin{proof}
Set $n\coloneqq\dim(Y)-1$. Let us prove that $\Tca=j_*\cS$ satisfies the three conditions of \cref{def:spherical with Serre}. Condition (a) is automatic, since $\tY$ is projective and $\widetilde{\cD}$ is a semiorthogonal component of $\Db(\tY)$ in the decomposition \eqref{eq:toeliminateSOD}. By \cref{lem:excep}, we have that
\begin{equation*}
    \Hom^\bullet(\Tca,\Tca)=\bC \oplus \bC[-2],
\end{equation*}
so condition (b) holds true. It remains to check condition (c). Recall that $\widetilde{\cD}$ has a Serre functor, given by \cref{lemma:serre1}(b); thus, by \cref{rem:spherical with serre}, it is enough to show that $\Serre_{\widetilde{\cD}}(\Tca)=\Tca[2]$. We have that
\[
  \Serre_{\widetilde{\cD}}(j_*\cS)=\mathbb{R}_{\widetilde{\cD}^\perp}  (\Serre_{\Db(\widetilde{Y})} (j_*\cS))= (\mathbb{R}_{j_*\cO_Q(-1)}\circ \dots \circ \mathbb{R}_{j_*\cO_Q(1-n)} \circ \Serre_{\Db(\widetilde{Y})})(j_*\cS),
\]
where $\Serre_{\Db(\widetilde{Y})}=\mathbb{T}_{\Canonical_{\widetilde{Y}}}\circ[n+1]$. Since, by the adjunction formula, we have the equality
\[
  j^*\Canonical_{\widetilde{Y}}=\Canonical_Q\otimes j^*\cO_{\widetilde{Y}}(-Q)= \cO_Q(-\dim(Q) + 1)=\cO_Q(1-n),
\]
we obtain
\[
  \Serre_{\Db(\widetilde{Y})} (j_*\cS) = j_*(\cS\otimes j^*\Canonical_{\widetilde{Y}})[n+1] = j_* \cS(1-n)[n+1].
\]
Now, using \cref{lem:R-mutate-odd}, we have
\begin{equation*}
    \mathbb{R}_{j_*\cO_Q(k)} (j_* \cS(k)[2-k])= j_* \cS(k+1)[2-k-1].
\end{equation*}
Proceeding inductively, we obtain
\[
  \Serre_{\widetilde{\cD}}(j_*\cS)=j_*\cS[2],
\]
proving the statement.
\end{proof}

\begin{proposition}\label{prop:seven}
If $Y$ is odd dimensional, 
then $\Tca$ is a $3$-spherical object in $\widetilde{\cD}$.
\end{proposition}

\begin{proof}
Again, since the category $\widetilde{\cD}$ is proper, condition (a) in \cref{def:spherical with Serre} is automatically satisfied. To check condition (b), recall that $\Tca$ sits in the exact triangle
\begin{equation}\label{eq:right-mute}
    j_* \cS' \rightarrow  j_* \cS''[2] \rightarrow \Tca.
\end{equation}
By definition we have $\Hom^\bullet(\Tca,j_*\cS'')=0$. Hence by applying $\Hom^\bullet(\Tca,-)$ to \eqref{eq:right-mute}  we obtain that 
\begin{equation*}
    \Hom^\bullet(\Tca,\Tca) = \Hom^\bullet(\Tca,j_* \cS'[1]),
\end{equation*}
and by applying $\Hom^\bullet(-,j_*\cS')$ to \eqref{eq:right-mute} we obtain
\begin{equation*}
    \Hom^\bullet(j_* \cS',j_*\cS') \leftarrow  \Hom^\bullet(j_* \cS''[2],j_*\cS') \leftarrow \Hom^\bullet(\Tca,j_*\cS').
\end{equation*}
We have by previous computations (\cf \cref{lem:excep} and \eqref{eq:push-S''-S'}) that
\begin{equation*}
    \Hom^\bullet(j_* \cS',j_*\cS')=\mathbb{C}, \quad\text{and}\quad  \Hom^\bullet(j_* \cS'',j_*\cS')=\mathbb{C}[-2].
\end{equation*}
We get from the long exact sequence that
\begin{equation*}
    \Hom^\bullet(\Tca,\Tca)=\bC \oplus \bC[-3].
\end{equation*}

To complete the proof we need to show that $\Serre_{\widetilde{\cD}}(\Tca) = \Tca[3]$. Using \cref{lemma:serre1}(b) with respect to the decomposition in \eqref{eq:SOD-R-mutated}, we have the factorization
\begin{equation*}
    \Serre_{\widetilde{\cD}}(\Tca)=\mathbb{R}_{\widetilde{\cD}^\perp}  (\Serre_{\Db(\widetilde{Y})} (\Tca))= (\mathbb{R}_{j_*\cS''} \circ \mathbb{R}_{j_*\cO_Q(-1)}\circ \dots \circ \mathbb{R}_{j_*\cO_Q(2-\dim(Y))}\circ \Serre_{\Db(\widetilde{Y})})(\Tca).
\end{equation*}
For the sake of keeping a lighter presentation, we write, by abuse of notation, $\Tca(k)$ in place of $\Tca\otimes \cO_{\widetilde{Y}}(-kQ)$, even though the object $\Tca$ does not belong to $\Db(Q)$.
As in \cref{prop:sperical-odd}, we have
\begin{equation*}
     \Serre_{\Db(\widetilde{Y})}(\Tca)=\Tca(2-\dim(Y))[\dim(Y)].
\end{equation*}
As $\mathbb{R}_{j_* \cO_Q(k)}$ is an exact functor, by \cref{lem:R-mutate-even} we have
\begin{align*}
    \mathbb{R}_{j_* \cO_Q}(\Tca) &=\cone(\mathbb{R}_{j_* \cO_Q} (j_* \cS')\rightarrow \mathbb{R}_{j_* \cO_Q} (j_* \cS''[2])) \\
    &=\cone(j_* \cS''(1)[-1]\rightarrow j_* \cS'(1)[1]) \\
    &= \Tca'(1)[-1], 
\end{align*}
where $\Tca'= \cone(j_* \cS''\rightarrow j_* \cS'[2])$. The arrow $j_*\cS''(1)[-1]\rightarrow j_*\cS'(1)[1]$ is nonzero (also similar for the arrows below), otherwise the object $\Serre_{\widetilde{\cD}}(\Tca)$ would become a direct sum of two objects, but this would contradict $\Hom^0(\Tca,\Tca)=\mathbb{C}$ as the Serre functor is an equivalence. Analogously, we obtain
\begin{equation*}
    (\mathbb{R}_{j_* \cO_Q(1)}\circ \mathbb{R}_{j_* \cO_Q})(\Tca) = \Tca(2)[-2],
\end{equation*}
and more generally
\begin{equation*}
    (\mathbb{R}_{j_* \cO_Q(k+1)}\circ \mathbb{R}_{j_* \cO_Q(k)})(\Tca(k)) = \Tca(k+2)[-2].
\end{equation*}
It follows that 
\begin{equation}
    (\mathbb{R}_{j_*\cO_Q(-2)}\circ \dots\circ \mathbb{R}_{j_*\cO_Q(2-\dim(Y))}) (\Tca(2-\dim(Y))[\dim(Y)]) = \Tca(-1)[3].
\end{equation}
Finally, we compute
\begin{equation*}
    \mathbb{R}_{j_*\cO_Q(-1)} (\Tca(-1)[3])=\Tca'[2],
\end{equation*}

and the last mutation
\begin{align*}
    \mathbb{R}_{j_*\cS''} (\Tca'[2])  &= \mathbb{R}_{j_*\cS''} (\cone(j_*\cS''\rightarrow j_* \cS'[2])[2]) \\
    &= \cone(\mathbb{R}_{j_*\cS''}(j_*\cS'')\rightarrow \mathbb{R}_{j_*\cS''}(j_* \cS'[2]))[2] \\
    &= \cone(0 \rightarrow \mathbb{R}_{j_*\cS''}(j_* \cS'[2]))[2] \\
    &=\cone(0 \rightarrow \Tca[1])[2] \\
    &= \Tca[3].
    \qedhere
\end{align*}
\end{proof}

\begin{remark}\label{rem:quasiprojective} This concludes the proof of \cref{thm:ns} in the case of a \emph{projective} variety $Y$ with an isolated nodal singularity $y$. We point out how to adjust the proofs when $Y$ is only supposed \emph{quasiprojective}. Let $Y'$ be a projective compactification of $Y$; by resolution of singularities, we can assume that $Y'$ is smooth outside $y$. We continue to denote by $\sigma \colon \tY \to Y$ the blow-up at the singular point, by $j \colon Q \to \tY$ the embedding of the exceptional divisor and by $n$ the dimension of $Q$. The variety $\tY$ is quasiprojective, and can be regarded as an open subset of the blow-up $\tY'$ of $Y'$ at $y$; we denote by $i \colon \tY \to \tY'$ the corresponding open immersion. 

Let us focus on the categorical aspects. We will denote as $\Db_Q(\tY)$ the full subcategory of $\Db(\tY)$ consisting of complexes topologically supported on $Q$; the functor $i_*$ embeds it as a full subcategory of $\Db(\tY')$.
Now, \cref{lem:excep} holds true even if $\Db(\tY)$ is not proper, because the functor $j_*$ has both left and right adjoints. From \cref{lem:R-mutate-odd} to \cref{prop:catresker}, all results hold without any change. In fact, by going through the proofs, from the Lefschetz decomposition 
\begin{equation*}
    \Db(Q)=\langle\cB_{n-1}(1-n) , \dots, \cB_1(-1), \cB_0 \rangle
\end{equation*} 
of \cref{thm:qsod} we deduce semiorthogonal decompositions not only for $\Db(\tY')$ and $\Db(\tY)$, but also for $\Db_Q(\tY)$: explicitly, we have
\begin{align*}
\Db(\tY') &= \langle (i \circ j)_*\cB_{n-1}(1-n) , \dots, (i \circ j)_*\cB_1(-1), \widetilde{\cD}' \rangle \\
\Db(\tY) &= \langle j_*\cB_{n-1}(1-n) , \dots, j_*\cB_1(-1), \widetilde{\cD} \rangle \\
\Db_Q(\tY) &= \langle j_*\cB_{n-1}(1-n) , \dots, j_*\cB_1(-1), \widetilde{\cD}_Q \rangle,
\end{align*}
where $\widetilde{\cD}'$ is defined as the left orthogonal of $\langle (i \circ j)_*\cB_{n-1}(1-n) , \dots, (i \circ j)_*\cB_1(-1)\rangle$ in $\Db(\tY')$, and $\widetilde{\cD}$ and $ \widetilde{\cD}_Q $ as the left orthogonal of $\langle j_*\cB_{n-1}(1-n) , \dots, j_*\cB_1(-1)\rangle$ in $\Db(\tY)$ and $\Db_Q(\tY)$, respectively. The categories $\widetilde{\cD}'$ and $\widetilde{\cD}$ provide categorical resolutions of $Y'$ and $Y$, respectively. Clearly $\widetilde{\cD}_Q=\widetilde{\cD} \cap \Db_Q(\tY)$, and we can easily verify that $\mathbb{R}_{\widetilde{\cD'}^\perp} \circ i_* = i_* \circ \mathbb{R}_{\widetilde{\cD_Q}^\perp}$ on $\Db_Q(\tY)$, so that $i_*\widetilde{\cD}_Q=\widetilde{\cD'} \cap i_*\Db_Q(\tY)$.

Consider now the classical generator $\Tca$ of $\ker(\sigma_*)$ given by \cref{prop:catresker}. We need to prove that it is spherical in the category $\widetilde{\cD}$. Since $\Tca$ belongs to $\widetilde{\cD}_Q$, the functors $\Hom^\bullet(\Tca, -)$ and $\Hom^\bullet(-, \Tca)$ on $\widetilde{\cD}$ take values in the category of finite-dimensional graded vector spaces, because so do they on $\Db(\tY)$ (see \cref{ex:serre quasiproj}); this shows that $\Tca$ satisfies condition (a) in \cref{def:spherical with Serre}. For the other two conditions, we can reason as follows. By \cref{ex:serre quasiproj} and \cref{lemma:Serre pair}, the pair $(\widetilde{\cD}_Q, \widetilde{\cD})$ has a Serre functor $\Serre$;
moreover, we have
\begin{multline*}
i_* \Serre (\Tca) = i_* \mathbb{R}_{\widetilde{\cD_Q}^\perp}(\Tca \otimes \omega_{\tY}[\dim(Y)]) 
= \mathbb{R}_{\widetilde{\cD'}^\perp} i_*(\Tca \otimes i^*\omega_{\tY'}[\dim(Y)]) \\
= \mathbb{R}_{\widetilde{\cD'}^\perp} (i_*\Tca \otimes \omega_{\tY'}[\dim(Y)])
= \mathbb{R}_{\widetilde{\cD'}^\perp} \circ \Serre_{\Db(\tY')}(i_*\Tca)
= \Serre_{\widetilde{\cD'}}(i_*\Tca).
\end{multline*}
From the isomorphism $i_* \Serre (\Tca)=\Serre_{\widetilde{\cD'}}(i_*\Tca)$ and the full faithfulness of $i_*$ on $\widetilde{\cD}_Q$ we deduce that conditions (c) and (b) in \cref{def:spherical with Serre} are satisfied by $\Tca$ in $\widetilde{\cD}$ if and only if they are satisfied by $i_*\Tca$ in $\widetilde{\cD}'$. Hence, the $k$-sphericalness of $\Tca$ in $\widetilde{\cD}$ is equivalent to the $k$-sphericalness of $i_*\Tca$ in $\widetilde{\cD'}$, which was proven in \cref{prop:sperical-odd} and \cref{prop:seven}. 
\end{remark}

This concludes the proof of \cref{thm:ns}.
In \cite[Definition~3.5]{Kldcr} another notion of crepancy was introduced in the categorical setting. A~categorical resolution $\widetilde{\cD}$ of $\cD$ is \emph{strongly crepant} if the relative Serre functor $\Serre_{\widetilde{\cD}/\cD}$ is isomorphic to the identity functor. We refer to \cite[Section~3]{Kldcr} for the definition of relative Serre functor. We only recall this notion in the case we consider, namely $\cD=\Db(Y)$ with a categorical resolution $\widetilde{\cD} \subset \Db(\widetilde{Y})$, where $\pi \colon \widetilde{Y} \to Y$ is a geometrical resolution of singularities: a functor $\Serre_{\widetilde{\cD}/Y} \colon \widetilde{\cD} \to \widetilde{\cD}$ is a relative Serre functor if for every $\mathcal{F}, \mathcal{G} \in \widetilde{\cD}$ there is a bifunctorial isomorphism
\begin{equation*}
\RHom(\pi_*\RHom(\mathcal{F}, \mathcal{G}), \cO_Y) \cong \pi_*\RHom(\mathcal{G}, \Serre_{\widetilde{\cD}/Y}(\mathcal{F})).
\end{equation*}
In the next proposition, we show that the weakly crepant categorical resolution $\widetilde{\cD}$ provided by \cref{prop:catres_part0} is not strongly crepant when the quasiprojective variety $Y$ with isolated nodal singularity has dimension at least~$4$. We stick with the notations of \cref{rem:quasiprojective}: $\sigma \colon \tY \to Y$ denotes the blow-up at the singular point, $Q$ its exceptional divisor, $\Db_Q(\tY)$ the full triangulated subcategory of $\Db(\tY)$ consisting of complexes topologically supported on $Q$, and $\widetilde{\cD}_Q = \widetilde{\cD} \cap \Db_Q(\tY)$. Recall that $\mathbb{T}_{\omega_{\tY}} \circ [\dim(Y)]$ is a Serre functor for the pair $(\Db_Q(\tY), \Db(\tY))$, and induces a Serre functor $\Serre$ for the pair $(\widetilde{\cD}_Q, \widetilde{\cD})$.

\begin{proposition} \label{prop:notstrcrep}
The categorical resolution $\widetilde{\cD}$ admits a relative Serre functor~$\Serre_{\widetilde{\cD}/Y}$, given by $\Serre_{\widetilde{\cD}/Y}=\mathbb{R}_{\widetilde{\cD}^\perp}\circ \mathbb{T}_{\cO((n-1)Q)}$.
\begin{enumerate}
\item For any $\Fca\in \widetilde{\cD}$ such that $j^*\Fca \in \langle \cO_Q \rangle$ we have $\Serre_{\widetilde{\cD}/Y}(\Fca)=\Fca$.
\item For any $\Fca\in\widetilde{\cD}_Q$ we have $\Serre_{\widetilde{\cD}/Y}(\Fca) = \Serre(\Fca)[-\dim(Y)]$.
\end{enumerate}
Therefore, if $\Tca$ is the classical generator of the kernel computed in \cref{prop:catresker}, we have $\Serre_{\widetilde{\cD}/Y}(\Tca)=\Tca[2-\dim(Y)]$ if $\dim(Y)$ is even and $\Serre_{\widetilde{\cD}/Y}(\Tca)=\Tca[3-\dim(Y)]$ if $\dim(Y)$ is odd. 
In particular, the categorical resolution $(\widetilde{\cD}, \sigma_*, \sigma^*)$ is not strongly crepant if $\dim(Y) >3$.
\end{proposition}
\begin{proof}
The relative canonical bundle of $\sigma$ is given by $\omega_{\widetilde{Y}/Y}=\cO_{\widetilde{Y}}((n-1)Q)$. By \cite[Proposition 4.7]{Kldcr}, the relative Serre functor $\Serre_{\tY/Y}=\mathbb{T}_{\omega_{\widetilde{Y}/Y}}$ of $\Db(\tY)$ induces a relative Serre functor $\Serre_{\widetilde{\cD}/Y}$ on $\widetilde{\cD}$; its explicit expression, as well as and part (a), can be found in \loccit. We prove part (b). 
Since $j^* \omega_{\widetilde{Y}/Y} = j^*\omega_{\widetilde{Y}} $, for any $\Gca \in \Db(Q)$ we have
\[
\Serre_{\widetilde{Y}/Y} (j_*\Gca) = j_* \Gca \otimes \omega_{\widetilde{Y}/Y} = j_* (\Gca\otimes j^*\omega_{\widetilde{Y}/Y}) = j_* (\Gca\otimes j^*\omega_{\widetilde{Y}}) = j_* \Gca \otimes \omega_{\widetilde{Y}} = \mathbb{T}_{\omega_{\tY}}(j_*\Gca).
\]
Hence, for any $\Fca \in \widetilde{\cD}_Q$, we have 
\[\Serre_{\widetilde{\cD}/Y} (\Fca)=(\mathbb{R}_{\widetilde{\cD}^\perp}\circ \Serre_{\widetilde{Y}/Y})(\Fca)=(\mathbb{R}_{\widetilde{\cD}^\perp}\circ \mathbb{T}_{\omega_{\tY}})(\Fca)=\Serre(\Fca)[-\dim(Y)].
\]
Therefore, as soon as $\dim(Y)>3$, the relative Serre functor $\Serre_{\widetilde{\cD}/Y}$ is not the identity on the spherical object $\Tca \in \widetilde{\cD}_Q$, so the categorical resolution $\widetilde{\cD}$ is not strongly crepant.
\end{proof}

We now deduce \cref{thm:geom-nodal-implies-abstract-nodal} from \cref{thm:ns}.
    
\begin{theorem}\label{thm:geom-nodal-implies-abstract-nodal2}
If $\cT$ is a geometric nodal category, then $\cT$ is an abstract nodal category, \ie there exists a categorical resolution $\sigma_*\colon\widetilde{\cD}\to\cT$ which is weakly crepant and whose kernel is classically generated by a single spherical object.
Furthermore, $\sigma_*\colon\widetilde{\cD}\to\cT$ is a localization up to direct summands.
\end{theorem}
\begin{proof}
By hypothesis, there exists a quasiprojective variety $Y$ which has only an isolated nodal singularity, and a semiorthogonal decomposition $\Db(Y)=\langle\cT,\cT'\rangle$ with $\cT^{\perf}$ not smooth.
We claim that this forces $\cT'^{\perf}$ to be smooth. For this, we look at the categories of singularities $\Dsg(Y)\coloneqq\Db(Y)/\Dperf(Y)$ and $\cT^\sg\coloneqq\cT/\cT^\perf$.
By \cite[§2, §3.3]{Orl:04} and \cite[Thm.~2.10]{Orl:11} we know that $\Dsg(Y)^{\oplus}\simeq\Dsg(\CC[z]/(z^2))$ if $\dim(Y)$ is even, and $\Dsg(Y)^{\oplus}\simeq\Dsg(\CC[x,y]/(xy))$ if $\dim(Y)$ is odd.
In the even dimensional case, following \cite[§3.3]{Orl:04}, one sees that there exist non-zero morphisms between any pair of non-zero objects in $\Dsg(\CC[z]/(z^2))$. Hence the full subcategory $\Dsg(Y)\subset\Dsg(\CC[z]/(z^2))$ admits no non-trivial semiorthogonal decomposition.
But we have the semiorthogonal decomposition $\Dsg(Y)=\langle\cT^\sg,\cT'^\sg\rangle$ by \cite[Prop.~1.10]{Orl:06}, so either $\cT^\sg=0$ or $\cT'^\sg=0$, as desired.
In the odd dimensional case, the category $\Dsg(\CC[x,y]/(xy))$ is equivalent to the category of $\ZZ/2\ZZ$-graded finite-dimensional vector spaces, where the shift functor swaps the graded pieces \cf \cite[Ex.~2.18]{KPS:21}, so we can conclude as before.

Let us assume that the dimension $\dim(Y)\geq 2$ is even, the proof of the odd dimensional case is similar. Then, by \cref{thm:ns}, we know that there is a weakly crepant categorical resolution $\sigma_*\colon \widetilde{\cD}\rightarrow \Db(Y)$ whose kernel $\ker(\sigma_*)$ is classically generated by a $2$-spherical object $\cS$.

Let us denote by $\imath\colon\cT'\rightarrow \Db(Y)$ the embedding functor. As $\cT'$ is admissible, it has a left adjoint functor $\imath^*$ and a right adjoint functor $\imath^!$. We know that $\imath(\cT')\subset \Dperf(Y)$ since $\cT'^{\perf}=\cT'$ by hypothesis.
Then we see that the functor $\sigma^*\circ\imath\colon\cT'\rightarrow\widetilde{\cD}$ is fully-faithful. Moreover, this functor has the right adjoint $\imath^!\circ\sigma_*$, thus making $\cT'$ an admissible subcategory of $\widetilde{\cD}$. So we can consider the semiorthogonal decomposition $\widetilde{\cD}=\langle\widetilde{\cT},\cT'\rangle$, where $\widetilde{\cT}\coloneqq\cT'^\perp$.

Now we claim that the restriction of $\sigma_*$ to $\widetilde{\cT}$ provides a categorical resolution that satisfies the conditions in \cref{def:anc}. First, if $\Fca\in\cT\cap\Dperf(Y)$ and $\Gca\in\cT'\subset\widetilde{\cD}$, then $\Hom^\bullet(\Gca,\sigma^*\Fca)=\Hom^\bullet(\sigma_*\Gca,\Fca)=0$, which implies that $\sigma^*$ maps $\cT^{\perf}$ to $\widetilde{\cT}$.
Second, if $\Fca\in\widetilde{\cT}$ and $\Gca\in\cT'$, then $\Hom^\bullet(\Gca,\sigma_*\Fca)=\Hom^\bullet(\sigma^*\Gca,\Fca)=0$, which implies that $\sigma_*$ maps $\widetilde{\cT}$  to $\cT$.
Regarding adjointness and weak crepancy, let $\Fca\in \cT^{\perf}$ and $\Gca\in\widetilde{\cT}$, and considering them as objects of $\Db(Y)$ and~$\widetilde{\cD}$, respectively, we see that $\Hom(\sigma^*\Fca,\Gca)=\Hom(\Fca,\sigma_*\Gca)$ and $\Hom(\Gca,\sigma^*\Fca)=\Hom(\sigma_*\Gca,\Fca)$.
In the same vein we have that $\id_{\cT^{\perf}}\rightarrow\sigma_*\sigma^*$ is an isomorphism.

By \cref{thm:ns} we know that \[\widetilde{\cD}/\langle\cS\rangle^{\oplus}\rightarrow \Db(Y)\] is an equivalence onto its dense image. Since $\sigma_*(\cS)=0$ and for $\Gca\in \cT'$ we have $\Hom(\sigma^*\Gca,\cS)=\Hom(\Gca,\sigma_*\cS)=0$, we see that $\cS\in \widetilde{\cT}$.
So, by the universal property of Verdier quotients, we can factor $\sigma_*|_{\widetilde{\cT}}\colon\widetilde{\cT}\to\cT$ via \[\overline{\sigma_*}\colon\widetilde{\cT}/\langle\Sca\rangle^{\oplus}\to\cT.\] Furthermore, by \cite[Lemma 1.1]{Orl:06} 
we have that the embedding $\widetilde{\cT}\subset\widetilde{\cD}$ descends to a fully-faithful functor $\widetilde{\cT}/\langle\Sca\rangle^{\oplus}\to\widetilde{\cD}/\langle\Sca\rangle^{\oplus}$, which implies that $\overline{\sigma_*}$ is fully-faithful. 

Using that $\sigma_*(\cT')\subset \cT'$, we see that $\im(\overline{\sigma_*})=\cT\cap\im(\sigma_*)$. We need to check that the idempotent completion of the latter is~$\cT$. 
Since $\sigma_*\sigma^*=\id_{\Dperf(Y)}$, we see that ${}^\perp\cT=\cT'\subset\im(\sigma_*)$, which implies that \[\im(\sigma_*)\cap\cT=\bL_{\cT'}(\im(\sigma_*)).\]
Since we know that $\im(\sigma_*)^\oplus=\Db(Y)$, we get
\begin{align*}
\cT=\cT'^\perp=\bL_{\cT'}(\Db(Y))=\ &\bL_{\cT'}(\im(\sigma_*)^\oplus) \\ &\subset(\bL_{\cT'}(\im(\sigma_*)))^\oplus=(\im(\sigma_*)\cap\cT)^\oplus.
\end{align*}
Hence, since $\cT$ is idempotent complete, we conclude $\cT=(\im(\sigma_*)\cap\cT)^\oplus$, as desired.
\end{proof}

%------------------------------------------------------------------------------%
% Section: Nodal Cubic Fourfolds
%------------------------------------------------------------------------------%

\section{Categorical resolutions of nodal cubic fourfolds} \label{sec:nodalcubic4}

In this section we focus on the special case of a cubic fourfold $Y \subset \mathbb{P}^5$ with a single isolated nodal singularity~$P\in Y$. Our goal is to prove \cref{prop:nodalc4} by applying \cref{thm:ns}.

\subsection{Geometric setting}
We first recall the geometric setting following \cite[Section~5]{Kcf}, which can be summarized in the diagram
\begin{equation}\label{eq:singular-cubic-diagram}
\begin{tikzcd}
& Q\ar[r, "j"]\ar[ld] & \tY\ar[ld, "\sigma"'] \ar[rd, "\pi"] & D \ar[l,
"i"'] \ar[rd, "p"]\\
P\ar[r] & Y & & \mathbb{P}^4 & S.\ar[l, ""']
\end{tikzcd}
\end{equation}
On the left hand side, the point $P$ is the nodal singular point and the morphism $\sigma\colon\tY\to Y$ is the blow-up of $Y$ at  
$P$; this yields the resolution of singularities $\tY$, whose exceptional divisor $Q$ is a smooth quadric of dimension~$3$. 
On the right hand side, the linear projection from $P$ induces a regular map $\pi\colon \tY\to \bP^4$, which can be shown to be the blow-up of $\bP^4$ along a smooth K3 surface $S$ that is a $(2,3)$-complete intersection, \cf\cite[Lemma~5.1]{Kcf}. We denote by $D$ the exceptional divisor of the map $\pi$, and write $j\colon Q\hookrightarrow\tY$ and $i\colon D\hookrightarrow\tY$ for the inclusions, as well as $p$ for the restriction $\pi|_D$.

Moreover, the restriction $\pi|_Q$ identifies the divisor $Q$ with the defining quadric of $S$ in $\bP^4$.
Also, the description of the map $\pi$ shows that the surface $S$ parametrizes the lines contained in $Y$ passing through~$P$. Such lines are contracted by the linear projection from $P$, and the~divisor~$D$ is the union of their strict transforms in $\tY$.

The following result clarifies the relation between $Q$, $D$, and $S$.

\begin{lemma}\label{lem:D-Q-square-cartesian}
The restriction $p|_Q=\pi|_{Q\cap D}$ of the projection map $\pi$ identifies $Q\cap D \subset \tY$ with the K3 surface $S$. In other words, $S$ is a retract of $D$ and the diagram
\[
\begin{tikzcd}
Q\ar[d, "j"] & S\ar[l, "t"'] \ar[d, "s"'] \\
\tY & D \ar[l, "i"'] \ar[u, bend right, "p"']
\end{tikzcd}
\]
is cartesian, where $t$ denotes the inclusion of $S$ into $Q$, and $s\colon S\isomto Q\cap D\hookrightarrow D$ denotes the inclusion into $D$.
\end{lemma}
\begin{proof}
Recall that $\pi|_Q$ is an isomorphism between $Q$ and the defining quadric of $S$ in $\bP^4$, and that $\pi(D)=S$. Therefore the intersection $Q\cap D$ is a closed subscheme of the pre-image $(\pi|_Q)^{-1}(S)$ in $Q$, which is a smooth K3 surface. Note that $Q\cap D$ is non-empty since each line in $Y$ passing through $P$ provides a point contained in $Q\cap D$. Then, by Krull's principal ideal theorem, the dimension of $Q\cap D$ is at least $2$ everywhere. We conclude that $Q\cap D$ coincides with the surface $(\pi|_Q)^{-1}(S)$. In other words, $\pi|_{Q\cap D}$ identifies $Q\cap D$ with the K3 surface $S$.
\end{proof}

\subsection{Computation of the kernel}
We work in the geometric situation summarized in diagram~\eqref{eq:singular-cubic-diagram}. Let $h$ be the class of a hyperplane in $\mathbb{P}^4$, and $H$ be the class of a hyperplane section of $Y\subset\mathbb{P}^5$. By abuse of notation, we use the same notation for their pullbacks to $\tY$.

Recall that $\langle \cO_Y,\cO_Y(H),\cO_Y(2H)\rangle$ is an exceptional sequence, so we have the semiorthogonal decomposition
\begin{equation} \label{eq:sodc4}
\Db(Y)=\langle \cA_Y, \cO_Y, \cO_Y(H), \cO_Y(2H) \rangle, 
\end{equation}
where $\cA_Y\coloneqq\langle \cO_Y, \cO_Y(H), \cO_Y(2H) \rangle^\perp$. Now we consider two semiorthogonal decompositions of $\Db(\tY)$ arising from the two geometric interpretations of the variety $\tY$, \cf\cite[(16), (17)]{Kcf}.
First, we apply \cref{prop:D-tilde-SOD} using the Lefschetz decomposition of $\Db(Q)$ from \cref{thm:qsod} to write
\[
    \Db(\tY)=\langle j_*\cO_Q(-2h),j_*\cO_Q(-h),\widetilde{\cD}\rangle,
\]
where $\widetilde{\cD}\coloneqq {}^\perp\langle 
j_*\cO_Q(-2h),j_*\cO_Q(-h) \rangle$ is a weakly crepant resolution of $\Db(Y)$. 
Then we consider the decomposition of $\widetilde{\cD}$ induced by that of $\Db(Y)$ in \eqref{eq:sodc4}.
As $Y$ has rational singularities, $\sigma^*:\Db(Y)\rightarrow \widetilde{\cD}$ is fully faithful, \cf \cite[Lemma~2.4]{Kldcr}, so the pullbacks of $\cO_Y$, $\cO_Y(H)$ and $\cO_Y(2H)$ along $\sigma$ are an exceptional sequence in $\widetilde{\cD}$, and we obtain
\begin{equation*}
    \widetilde{\cD}=\langle \widetilde{\cA}_Y,\cO_{\tY},\cO_{\tY}(H),\cO_{\tY}(2H) \rangle.
\end{equation*}
Substituting this in the decomposition above, we get
\begin{equation}
    \Db(\tY)=\langle j_*\cO_Q(-2h), j_*\cO_Q(-h),\widetilde{\cA}_Y,\cO_{\tY},\cO_{\tY}(H),\cO_{\tY}(2H)\rangle.
\end{equation}
Note that the residual category $\widetilde{\cA}_Y$ is a weakly crepant resolution of $\cA_Y$ by \cite[Lemma~5.8]{Kcf}.
On the other hand, since $\tY$ is the blow-up of $\bP^4$ along the K3 surface $S$, we have by Orlov's blow-up formula \cite{Orlov} that
\begin{equation}
    \Db(\tY) =\langle\Phi(\Db(S)), \cO_{\tY}(-3h),\cO_{\tY}(-2h),\cO_{\tY}(-h),\cO_{\tY},\cO_{\tY}(h)\rangle,
\end{equation}
where $\Phi \colon \Db(S) \to \Db(\widetilde{Y})$ is given by $\Phi=\bT_{\cO_{\widetilde{Y}}(D)} \circ i_* \circ p^*$.
Recall that $\bT_{\cO_{\widetilde{Y}}(D)}$ denotes the functor which twists by $\cO_{\widetilde{Y}}(D)$.

Using a series of mutations, one may relate these two decompositions and show that there is an equivalence $\Phi''\colon \Db(S)\isomto \widetilde{\cA}_Y$, \cf \cite[Corollary~5.7]{Kcf}, so $\Db(S)$ is also a weakly crepant categorical resolution of $\cA_Y$. 
The equivalence is explicitly given by
\[
  \Phi''= \mathbb{R}_{\cO_{\tY}(-h)} \circ \mathbb{R}_{\cO_{\tY}(-2h)} \circ \mathbb{T}_{\cO_{\tY}(D-2h)} \circ i_* \circ p^*.
\]

Now, applying \cref{prop:catresker}, the weakly crepant categorical resolution of $\Db(Y)$ given by $\widetilde{\cD}$ together with the restrictions of $\sigma_*$ and $\sigma^*$ has kernel classically generated by $j_*\cS$. We show that the latter is also a classical generator of the kernel of $\widetilde{\cA}_Y \to \cA_Y$.
Since $\widetilde{\cA}_Y$ is an admissible subcategory, it is in particular thick, so it suffices to prove the following lemma.
\begin{lemma} \label{lemma:sintildeA}
The object $j_*\cS$ lies in $\widetilde\cA_Y$.
\end{lemma}
\begin{proof}
By \cref{prop:catresker}, we have that $j_*\cS$ lies in $\widetilde{\cD}={}^\perp \left< j_*\cO_Q(-2h), j_*\cO_Q(-h) \right>$. It now suffices to verify that $j_*\cS\in \left< \cO_{\tY}, \cO_{\tY}(H), \cO_{\tY}(2H) \right>^\perp$.

Note that for all $k\in\bZ$ we have
\[
\Hom^\bullet_{\tY}(\cO_{\tY}(kH),j_*\cS)=
\Hom^\bullet_Q(\cO_Q,\cS)=0,
\]
since $Q$ is the exceptional divisor of
$\sigma$ and the line bundle $\cO_{\tY}(kH)$ pulls back to the trivial line bundle on $Q$.
\end{proof}

Next we describe $j_*\cS$ as an object in $\Db(S)$ using the left adjoint of $\Phi''$. The latter has been computed in \cite[Remark~5.9]{Kcf}, but beware of a misprint in \loccit, so we provide here its correct expression.

\begin{proposition}
The left adjoint of $\Phi''$ is 
\[
  \Psi=p_* \circ i^* \circ \bT_{\cO_{\tY}(-3h+D)[1]} \circ \bL_{\cO_{\tY}(3h-D)} \circ \bL_{\cO_{\tY}(4h-D)}.
\]
\end{proposition}

\begin{proof}
Recall that if $\Eca$ is an exceptional object in $\Db(X)$, where $X$ is a smooth projective variety, then the functor $\bR_{\Serre_X(\Eca)}$ is right adjoint to $\bL_\Eca$, where $\Serre_X$ is the Serre functor of $\Db(X)$. Using this fact and that the canonical class of
$\tY$ is $-5h+D$, we obtain
\begin{align*}
\Hom_{\tY}(\Aca, \Phi''(\Bca))
=& \Hom_{\tY}(\Aca, (\bR_{\cO_{\tY}(-h)} \circ \bR_{\cO_{\tY}(-2h)} \circ \bT_{\cO_{\tY}(D-2h)} \circ i_* \circ p^*)(\Bca)) \\
=& \Hom_D((i^* \circ \bT_{\cO_{\tY}(2h-D)} \circ \bL_{\cO_{\tY}(3h-D)} \circ \bL_{\cO_{\tY}(4h-D)})(\Aca), p^* \Bca).
\end{align*}
Now we need to compute the left adjoint of $p^*$. The canonical bundle of $D$
is by adjunction
\[
\Canonical_D=(\Canonical_{\tY}\otimes\cO(D))|_D = \big(\cO(-5h+D)\otimes\cO(D)\big)|_D.
\]
So we have $\Canonical_D = i^*\cO_{\tY}(-5h+2D)$. We compute the left adjoint of $p^*$ using Grothendieck--Verdier duality
\[
\begin{aligned}
\Hom_D(\Fca, p^* \Bca)
&=\Hom_D(p^* \Bca, \Fca\otimes \Canonical_D[3])^{\vee} \\
&=\Hom_D(p^* \Bca, \Fca(-5h+2D)[3])^{\vee} \\
&=\Hom_S(\Bca, p_* \Fca(-5h+2D)[3])^{\vee} \\
&=\Hom_S(p_* \Fca(-5h+2D)[3], \Bca\otimes \Canonical_S[2]) \\
&=\Hom_S(p_* \Fca(-5h+2D)[1], \Bca).
\end{aligned}
\]
This shows that the left adjoint of $p^*$ is $p_*\circ \bT_{i^*\cO_{\tY}(-5h+2D)[1]}$. Putting everything together, we obtain
\[
\begin{aligned}
&\Hom_D((i^* \circ \bT_{\cO_{\tY}(2h-D)} \circ \bL_{\cO_{\tY}(3h-D)} \circ \bL_{\cO_{\tY}(4h-D)})(\Aca), p^*\Bca) \\
=& \Hom_S((p_* \circ \bT_{i^*\cO_{\tY}(-5h+2D)[1]} \circ i^*\circ \bT_{\cO_{\tY}(2h-D)} \circ \bL_{\cO_{\tY}(3h-D)} \circ \bL_{\cO_{\tY}(4h-D)})(\Aca), \Bca) \\
=& \Hom_S((p_* \circ i^* \circ \bT_{\cO_{\tY}(-5h+2D)[1]} \circ \bT_{\cO_{\tY}(2h-D)} \circ \bL_{\cO_{\tY}(3h-D)} \circ \bL_{\cO_{\tY}(4h-D)})(\Aca), \Bca) \\
=& \Hom_S((p_* \circ i^* \circ \bT_{\cO_{\tY}(-3h+D)[1]} \circ \bL_{\cO_{\tY}(3h-D)} \circ \bL_{\cO_{\tY}(4h-D)})(\Aca), \Bca),\\
\end{aligned}
\]
and thus
\[
\Psi=p_* \circ i^* \circ \bT_{\cO_{\tY}(-3h+D)[1]} \circ \bL_{\cO_{\tY}(3h-D)} \circ \bL_{\cO_{\tY}(4h-D)},
\]
proving the statement.
\end{proof}

We now identify $\Psi(j_*\cS)$ as the restriction of the spinor bundle on $Q$ to~$S$.

\begin{proposition} \label{prop:SinS}
We have that $\Psi (j_*\cS) =  t^*\cS$, where $t\colon S\rightarrow Q$ is the inclusion of $S$ into the quadric $Q$ which is embedded in $\mathbb{P}^4$ via $\pi\circ j$.
\end{proposition}
\begin{proof}
Note that the first two mutations in the formula of $\Psi$ have no effect, since we have, using the relation $D=3h-H$,
\[
\begin{aligned}
\Hom^\bullet_{\widetilde{Y}}(\cO_{\widetilde{Y}}(4h-D),j_*\cS)
&=\Hom^\bullet_{Q}(j^*\cO_{\widetilde{Y}}(4h-D),\cS) \\
&=\Hom^\bullet_{Q}(j^*\cO_{\widetilde{Y}}(h+H),\cS) \\
&=\Hom^\bullet_{Q}(j^*\cO_{\widetilde{Y}}(h),\cS) \\
&=\Hom^\bullet_{Q}(\cO_Q(h),\cS) =0,
\end{aligned}
\]
and similarly for the second mutation. Applying $\bT_{\cO_{\tY}(-3h+D)}$, we get
\[
  j_*\cS\otimes \cO_{\tY}(D-3h)=j_*\cS\otimes \cO_{\tY}(-H)=j_*(\cS\otimes j^*\cO_{\tY}(-H))=j_*(\cS\otimes \cO_Q)=j_* \cS.
\]
The last step is to calculate $p_*i^*j_*\cS$. We consider the diagram
\begin{equation}\label{eq:D-Q-complete-intersection}
\begin{tikzcd}
Q\ar[d, "j"] & S\ar[l, "t"'] \ar[d, "s"'] \ar[dd, "\id", bend left] \\
\tY\ar[d, "\pi"] & D \ar[l, "i"'] \ar[d, "p"']\\
\mathbb{P}^4 & S\ar[l, ""'] 
\end{tikzcd}
\end{equation}
where the upper square is cartesian by \cref{lem:D-Q-square-cartesian}. We prove $i^*j_*=s_*t^*$ by checking the conditions of the base-change result \cref{prop:com}. 
Indeed, as $Q$ and $\tY$ are smooth, they are Cohen--Macaulay. The closed immersion $i$ of the exceptional divisor of the blow-up is by its very nature a local complete intersection in $\tY$. Finally, we have $\codim_{\tY}(D)=\codim_{D}(S)=1$.
Thus we obtain
\[
  p_*i^*j_*\cS=p_*s_*t^*\cS=t^*\cS.
  \qedhere
\]
\end{proof}

We conclude this section with the proof of \cref{prop:nodalc4}.

\begin{proof}[Proof of \cref{prop:nodalc4}]
By \cref{thm:ns} and \cref{lemma:sintildeA} we have that the kernel of $\Db(S) \to \cA_Y$ is classically generated by $j_*\cS$. Then the statement follows from \cref{prop:SinS}.   
\end{proof}

\bibliographystyle{alpha}
\bibliography{main.bib}

\end{document}